\documentclass{amsart}
\usepackage{latexsym,amsxtra,amscd,ifthen,amsmath,color, multicol,hyperref,mathdots}
\usepackage{amsfonts}
\usepackage{verbatim}
\usepackage{amsmath}
\usepackage{amsthm}
\usepackage{amssymb}
\usepackage[notcite,notref,final]{showkeys}
 \usepackage[all,cmtip]{xy}

\numberwithin{equation}{section}

\theoremstyle{plain}
\newtheorem{theorem}{Theorem}[section]
\newtheorem{lemma}[theorem]{Lemma}

\newtheorem{proposition}[theorem]{Proposition}

\newtheorem{corollary}[theorem]{Corollary}

\theoremstyle{definition}
\newtheorem{definition}[theorem]{Definition}
\newtheorem{example}[theorem]{Example}
\newtheorem{hypothesis}[theorem]{Hypothesis}
\newtheorem{notation}[theorem]{Notation}

\newtheorem{remark}[theorem]{Remark}
\newtheorem{question}[theorem]{Question}

\makeatletter              
\let\c@equation\c@theorem  
\makeatother

\newcommand*{\ee}{\ensuremath{\epsilon}}  
\newcommand{\kk}{\Bbbk}
\newcommand{\DD}{\mathbb{D}}
\newcommand{\E}{\mathbb{E}}
\newcommand{\F}{\mathbb{F}}
\newcommand{\A}{\mathbb{A}}
\newcommand{\B}{\mathbb{B}}
\newcommand{\I}{\mathbb{I}}
\newcommand{\J}{\mathbb{J}}
\newcommand{\M}{\mathbb{M}}
\newcommand{\N}{\mathbb{N}}
\newcommand{\X}{\mathbb{X}}
\newcommand{\Y}{\mathbb{Y}}
\newcommand{\T}{\mathbb{T}}
\newcommand{\LL}{\mathbb{L}}
\newcommand{\U}{\mathbb{U}}

\newcommand{\D}{{\sf D}}
\newcommand{\Ext}{\text{Ext}}

\newcommand{\gldim}{\text{gl.dim}}

\newcommand{\Diag}{\text{Diag}}  

\newcommand{\mc}{\mathcal}

\begin{document}

\title[On quantum groups associated to non-Noetherian regular algebras of dimension 2]
{On quantum groups associated to non-Noetherian regular algebras of dimension 2}

\author{Chelsea Walton and Xingting Wang}

\address{Department of Mathematics, Temple University, Philadelphia, Pennsylvania 19122
USA}
\email{notlaw@temple.edu}

\address{Department of Mathematics, Temple University, Philadelphia, Pennsylvania 19122
USA}
\email{xingting@temple.edu}

\bibliographystyle{abbrv}       

\begin{abstract}
We investigate homological and ring-theoretic properties of universal quantum linear groups that coact on Artin-Schelter regular algebras $A(n)$ of global dimension 2, especially with central homological codeterminant (or central quantum determinant). As classified by Zhang, the algebras $A(n)$ are connected $\N$-graded algebras that are finitely generated by $n$ indeterminants of degree $1$, subject to one quadratic relation. In the case when the homological codeterminant of the coaction is trivial, we show that the quantum group of interest, defined independently by Manin and by Dubois-Violette and Launer, is Artin-Schelter regular of global dimension 3 and also skew Calabi-Yau (homologically smooth of dimension 3). For central homological codeterminant, we verify that the quantum groups are Noetherian and have finite Gelfand-Kirillov dimension precisely when the corresponding comodule algebra $A(n)$ satisfies these properties, that is, if and only if $n=2$. We have similar results for arbitrary homological codeterminant if we require that the quantum groups are involutory. We also establish conditions when Hopf quotients of these quantum groups, that also coact on $A(n)$, are cocommutative. 

\end{abstract} 

\subjclass[2010]{16E65, 16S38, 16T05}
\keywords{Artin-Schelter regular algebra, non-Noetherian, quantum linear group, homological codeterminant}

\maketitle


\setcounter{section}{-1}



\section{Introduction}

Let $\kk$  be an algebraically closed field of characteristic zero. Let an unadorned $\otimes$ denote $\otimes_{\kk}$, and all algebras, Hopf algebras, etc., in this work are over $\kk$. 
Let $R$ be an {\it Artin-Schelter \textnormal{(}AS\textnormal{)} regular algebra}, that is, a graded homological analogue of a polynomial algebra; see Definition~\ref{def:ASreg}. 
\medskip

\begin{center}
{\bf Throughout,  all AS regular algebras are assumed to be generated in degree 1.} 
\end{center}
\medskip
It has been long suspected that: 
\[
(\star) \hspace{1.1in} 
\begin{array}{c}
\text{{\it The universal quantum \textnormal{(}linear\textnormal{)} groups that coact on $R$ possess the same }} \\ 
\text{{\it homological and ring-theoretic properties of $R$.}}
\end{array}  \hspace{1.1in}
\]

\noindent See, e.g., Manin's question in \cite[Introduction]{AST}.  The statement ($\star$) has been verified for $R$ being a skew polynomial algebra or a Noetherian AS regular algebra of global dimension 2, c.f. \cite[Sections~I.2 and~II.9]{book:BrownGoodearl}. It is important to note that in these cases the {\it homological codeterminant} [Definition~\ref{def:hdet}]  (also known as the {\it quantum determinant}) of the coaction on $R$ is a central element in the quantum group; c.f. \cite[Exercise~I.2.E]{book:BrownGoodearl}.

The goal of this work is to verify ($\star$) for the class of AS regular algebras of global dimension 2 that are not necessarily Noetherian.
These algebras were classified by Zhang in \cite{Zhang:nonnoeth}, and have the following presentation:
\begin{equation} \label{eq:A}
A(n):=A(\E) := A(n, \E) := \kk\langle v_1, \dots, v_n \rangle/ \left(\textstyle \sum_{1 \leq i,j \leq n} e_{ij} v_i v_j \right),
\end{equation}
for $\E:=(e_{ij}) \in GL_n(\kk)$, with $n \geq 2$. Here, we do not assume finite Gelfand-Kirillov (GK) dimension for the definition of Artin-Schelter regularity. See Section~\ref{sec:ASreg} for more details.

We first consider two universal quantum linear groups that coact on the algebra $A(\E)$ with {\it trivial} homological codeterminant: the Hopf algebra $\mc{O}_{A(\E)}(SL)$ due to Manin \cite{Manin:QGNCG} (see Section~\ref{subsec:OA}), and the Hopf algebra $\mc{B}(\E^{-1})$ due to Dubois-Violette and Launer \cite{DVL}  (see Section~\ref{subsec:B(E)}).
In fact, we obtain that $$\mc{O}_{A(\E)}(SL) \cong \mc{B}(\E^{-1}) \quad \quad \text{ as Hopf algebras},$$ see Corollary~\ref{C:SLGL}. Note that for $q\in \kk^\times$, we have that $\mc{B}(\E)$ and the quantum linear group $\mc{O}_q(SL_2(\kk))$ are ``co-Morita equivalent" by \cite[Theorem 1.1]{Bichon:B(E)}. See Example~\ref{ex:B(Aq)} for details on the latter quantum group.

In Section~\ref{sec:hom}, we establish the following homological properties of $\mc{O}_{A(\E)}(SL)$; these follow essentially from work of Bichon  on the homological properties of $\mc{B}(\E)$ \cite{Bichon}.

\begin{theorem}[Propositions~\ref{prop:AS},~\ref{prop:skewCY} and Corollaries~\ref{cor:CY} and~\ref{C:SLGL}]
\label{thm:homintro}
Let $A:= A(n,\E)$ be an Artin-Schelter regular algebra of global dimension $2$. We have the following statements.
\begin{enumerate}
\item $\mc{O}_{A}(SL)$ is Artin-Schelter regular of global dimension 3 in the sense of \cite[Definition 1.2 (omitting Noetherianity)]{BrownZhang:Dualizing}.
\item $\mc{O}_{A}(SL)$ is skew Calabi-Yau \textnormal{(}homologically smooth of dimension 3\textnormal{)} \textnormal{[Definition~\ref{def:skewCY}]}.
\item If $\E$ is symmetric, or if $\E$ is skew-symmetric with $n$ even, then $\mc{O}_{A}(SL)$ is both Calabi-Yau and involutory \textnormal{(}that is, the square of the antipode is the identity\textnormal{)}. 
\end{enumerate}
\end{theorem}

In the philosophy of ($\star$), these results are expected since $A(\E)$ is AS regular of global dimension 2 by \cite[Theorem~0.1]{Zhang:nonnoeth}, and hence, skew Calabi-Yau  by \cite[Lemma~1.2]{RRZ}. The algebra $A(n,\E)$ is also Calabi-Yau  (or {\it $1$-Nakayama} [Definition~\ref{def:rNaka}]) if and only if $n$ is even with $\E$ skew-symmetric \cite[Proposition~3.4]{Berger}. Moreover, $A(\E)$ is ($-1$){\it -Nakayama} if and only if $\E$ is symmetric; refer to Lemma~\ref{lem:NakAutom}.

In Section~\ref{sec:ring}, we establish ring-theoretic properties of $\mc{O}_{A(\E)}(SL)$ (or $\mc{B}(\E^{-1})$), which provide further evidence of ($\star$) for $R = A(\E)$. These properties can also be extended to variants of another one of Manin's universal quantum linear groups, $\mc{O}^{{c}}_{A(\E)}(GL)$ and $\mc{O}_{A(\E)}(GL/S^2)$, and extended to quantum groups introduced by Mrozinski,   $\mc{G}(\X,\Y)$ for $\X,\Y \in GL_n(\kk)$ \cite{Mrozinski}. Here, $\mc{O}^{{c}}_{A(\E)}(GL)$ and $\mc{G}(\E^{-1},\E)$ both coact on $A(\E)$ with {\it central} homological codeterminant. On the other hand, $\mc{O}_{A(\E)}(GL/S^2)$ and $\mc{G}(\E^{-1},\E^T)$ are involutory, and both coact on $A(\E)$ with {\it arbitrary} homological codeterminant. See Sections~\ref{subsec:OA}~and~\ref{subsec:B(E)} for more details. Likewise, 
$$\mc{O}^{c}_{A(\E)}(GL) \cong \mc{G}(\E^{-1},\E) \quad \quad \text{and} \quad \quad 
\mc{O}_{A(\E)}(GL/S^2) \cong \mc{G}(\E^{-1},\E^T) \quad \quad \text{ as Hopf algebras},$$ 
by Corollary~\ref{C:SLGL}. Moreover, for  $q\in \kk^\times$, we have that $\mc{G}(\X,\Y)$ and the quantum linear group $\mc{O}_q(GL_2(\kk))$ are ``co-Morita equivalent" by \cite[Theorem 1.3]{Mrozinski}. We take $\mc{O}_{A}(SL/S^2)$ to be the universal involutory quantum linear group that coacts on $A$ with trivial homological codeterminant. Consider the following result.

\begin{theorem}[Theorem~\ref{thm:ring} and \cite{Zhang:nonnoeth}] \label{thm:ringintro}
Let $A:=A(n,\E)$ be an Artin-Schelter  regular algebra of global dimension $2$. Then, we have that the following statements are equivalent.
\begin{enumerate}
\item $A$ is Noetherian.
\item $A$ has finite Gelfand-Kirillov dimension.
\item $\mc{O}_{A}(SL/S^2)$,  $\mc{O}_{A}(GL/S^2)$, $\mc{O}_{A}(SL)$, and $\mc{O}^c_{A}(GL)$ are Noetherian.
\item $\mc{O}_{A}(SL/S^2)$,  $\mc{O}_{A}(GL/S^2)$, $\mc{O}_{A}(SL)$, and $\mc{O}^c_{A}(GL)$ have finite Gelfand-Kirillov dimension.
\item $n=2$.
\end{enumerate}
\end{theorem}

The theorem above is expected as the equivalence of (a), (b), and (e) was established in \cite[Theorem~0.1]{Zhang:nonnoeth}.  As shown in the proof of Theorem~\ref{thm:ring}, if $n=2$, then both of the quantum groups $\mc{O}^{c}_{A}(GL)$ and $\mc{O}_{A}(GL/S^2)$ are a localization of an iterated Ore extension in four variables at the homological codeterminant.

It is natural that the introduction of the central homological codeterminant condition is accompanied by usage of the involutory condition, as there are {\it homological identities} that relate these two notions; see e.g. \cite[Theorem~0.1]{CWZ:Nakayama}.

\begin{remark}
In fact, in many of the applications of \cite[Theorem~0.1]{CWZ:Nakayama} (or of its generalization \cite[Theorem~4.3]{RRZ2}), one can obtain a more general result (on  Hopf coactions on AS regular algebras) by replacing the trivial homological codeterminant hypothesis with the assumption that the homological codeterminant of the coaction is central. For example, this applies to \cite[Theorems~0.4 and~0.6]{CWZ:Nakayama} and \cite[Theorem~0.5]{LMZ}.
\end{remark}

Motivated by the discussion above, we revise the philosophy ($\star$) as follows.
\[
(\star \star) \hspace{0.8in} 
\begin{array}{c}
\text{{\it The universal quantum \textnormal{(}linear\textnormal{)} groups $\mc{Q}$ that coact on $R$}}\\
\text{{\it  with either \textnormal{(i)} central quantum determinant or \textnormal{(ii)}  $\mc{Q}$ being involutory,  }} \\ 
\text{{\it possess the same homological and ring-theoretic properties of $R$.}}
\end{array} \hspace{1.1in}
\]

On another note, we have that the AS regular algebras $A(\E)$ are all graded coherent domains by \cite[Theorem~1.4]{Piontkovskii} and \cite[Theorem~0.1]{Zhang:nonnoeth}. This prompts the question below.

\begin{question}
Let $A$ be an AS regular algebra of global dimension $2$. Are $\mc{O}_{A}(SL/S^2)$, $\mc{O}_{A}(GL/S^2)$ $\mc{O}_{A}(SL)$,  and $\mc{O}^{c}_{A}(GL)$ coherent domains?
\end{question}

Examples of the quantum linear groups above are presented at the end of Section~\ref{sec:quantumlinear} for the case when $n=2$. In particular, we show  
that the main result of \cite{PavelWalton}, that finite dimensional semisimple Hopf actions on commutative domains must factor through a finite group  action, fails in the infinite dimensional setting in general. We refer the reader to \cite{Chirvasitu, DGJ, Goswami:quadratic} for conditions when the infinite dimensional analogue of \cite[Theorem~1.3]{PavelWalton} holds.

Nevertheless in Section~\ref{sec:finite}, we establish general results on inner-faithful (not necessarily finite dimensional) Hopf coactions on $A(\E)$. We emphasize that: 
\medskip
\begin{center}
{\bf Throughout, we assume that Hopf coactions on $A(\E)$ preserve the grading of $A(\E)$.}
\end{center}
\medskip
A key part of the investigation in Section~\ref{sec:finite} is the requirement that $A(n,\E)$ is {\it central Nakayama} (resp., {\it power central Nakayama}), that is, the requirement that the centralizer algebra for (resp., powers of) the matrix representing the Nakayama automorphism of $A(n,\E)$ is a commutative subalgebra of  the matrix algebra $M_n(\kk)$; see Definitions~\ref{def:cocom} and~\ref{def:centralNak}. We prove the following result for a general AS regular algebra $R$ of global dimension $d$ satisfying the additional hypothesis:

\begin{hypothesis} \label{hyp}
Let $\mu_R$ be the Nakayama automorphism of $R$ and let $\mu_E$ be the Nakayama automorphism of its Yoneda algebra $E:=\Ext^*_R(\!_R\kk,\!_R\kk)$. Assume the following identity $\mu_E|_{E_1}=(-1)^{d+1}(\mu_R|_{R_1})^*$.
\end{hypothesis}

\begin{theorem} [Theorem~\ref{thm:centralNak}]
Let $R$ be an AS regular algebra satisfying Hypothesis~\ref{hyp}. Let $H$ be a Hopf algebra with antipode of finite order coacting on $R$. 
Suppose that  the $H$-coaction on $R$ is inner-faithful with homological codeterminant ${\sf D} \in H$. 
If one of the following conditions~hold:
\begin{enumerate}
\item[(i)] $R$ is central Nakayama, $H$ is involutory, and ${\sf D}$ is central; or
\item[(ii)] $R$ is power central Nakayama and ${\sf D}^m$ is central for some $m \geq 1$,
\end{enumerate}
then $H$ is cocommutative.  If, further, $H$ is finite dimensional, then $H$ must be a group algebra.
\end{theorem}

Hypothesis~\ref{hyp} holds when $R$ is $N$-Koszul AS regular or is Noetherian AS regular, by \cite[Theorem~6.3]{BergerMarconnet} and \cite[Theorem~4.2(3)]{RRZ2}, respectively. 
It is conjectured that Hypothesis~\ref{hyp} holds for all AS regular algebras generated in degree 1 \cite[Remark~4.2]{CWZ:Nakayama}. A consequence of the theorem above is given below.

\begin{corollary}[Corollary~\ref{cor:cocom}]
Let $H$ be a Hopf algebra with antipode of finite order coacting on $A(\E)$ inner-faithfully with  homological codeterminant ${\sf D}\in H$. If $\E$ is generic, and if either:
\begin{enumerate}
\item[(i)] $H$ is involutory, and ${\sf D}$ is central; or
\item[(ii)] ${\sf D}^m$ is central for some $m \geq 1$,
\end{enumerate}
 then $H$ is cocommutative. If, further, $H$ is finite dimensional, then $H$ must be a group algebra.
\end{corollary}


\section{Background material} 
We provide background on the following topics: Artin-Schelter regular algebras, specifically of global dimension 2 in Section~\ref{sec:ASreg}; Nakayama automorphisms and skew Calabi-Yau algebras in Section~\ref{sec:Nak}; and Hopf algebra coactions and homological codeterminant in Section~\ref{sec:Hopf}. 

\subsection{Artin-Schelter regular algebras $A(\E)$ of global dimension 2}
\label{sec:ASreg}

The algebras that are the focus of this work are the {\it Artin-Schelter regular algebras} defined below.

\begin{definition} \label{def:ASreg}
Let $R$ be a connected $\mathbb{N}$-graded algebra. Namely, $R= \bigoplus_{i \geq 0} R_i$ with $R_i R_j \subseteq R_{i+j}$ and $R_0=\kk$.
We say that $R$ is
{\it Artin-Schelter Gorenstein} (or is {\it AS Gorenstein}) if
\begin{enumerate}
\item[(i)] $R$ has finite injective dimension $d < \infty$ on both sides, and
\item[(ii)] $\Ext^i_R(_R\kk,_RR) \cong \Ext^i_R(\kk_R,R_R)= \delta_{i,d} ~\kk$.
\end{enumerate}
If further
\begin{enumerate}
\item[(iii)] $R$ has finite global dimension,
\end{enumerate}
then $R$ is called {\it Artin-Schelter regular (or {\it AS regular}) of global dimension $d$}.
\end{definition}

In this paper, we are not assuming that $R$ has finite Gelfand-Kirillov dimension as in the standard definition of AS regularity. Similarly, we say a (ungraded) Hopf algebra $H$ is {\it AS Gorenstein} (resp., {\it AS regular}) if it satisfies the conditions (i)-(ii) (resp., the conditions (i)-(iii)) above, where $\kk$ is considered to be the trivial $H$-module. See \cite[Definition 1.2 (omitting Noetherianity)]{BrownZhang:Dualizing}. The connected $\mathbb{N}$-graded AS regular algebras of global dimension 2 are classified  by Zhang as follows.

\begin{theorem} \label{thm:Zhang} \cite[Theorem~0.1]{Zhang:nonnoeth}  A connected $\mathbb{N}$-graded algebra $A$ is AS regular of global dimension 2 if and only if $A$ is isomorphic to $A(n, \E)$ from \eqref{eq:A} for some $n \geq 2$, with $\E$ of full rank.
\end{theorem}

To work with isomorphism classes of $A(n,\E)$, we use the following results.

\begin{lemma} \cite[End of Section 5]{Berger} \cite[Proposition~3.1]{BergerPichereau} 
\label{lem:isom}
We have that $A(n, \E) \cong A(n, \E')$ if and only if $\E$ and $\E'$ are congruent, that is to say, if and only if $\E' = \mathbb{P}^T \E \mathbb{P}$ for some $\mathbb{P} \in GL_n(\kk)$. \qed
\end{lemma}

\begin{example} \label{ex:01 10}
We have that $\kk\langle u,v \rangle/(u^2 + v^2)$ is isomorphic to $\kk\langle u,v \rangle/(uv+vu)$ since 
${\footnotesize \begin{pmatrix} 1 & 0\\ 0& 1 \end{pmatrix}}$ and ${\footnotesize \begin{pmatrix} 0 & 1\\ 1& 0 \end{pmatrix}}$ are congruent via $\mathbb{P} = {\footnotesize \begin{pmatrix} \frac{1}{2} & \frac{-i}{2} \\ 1& i \end{pmatrix}}$ with $i^2=-1$.
\end{example}

To employ Lemma~\ref{lem:isom}, we recall the canonical matrices of bilinear forms. 

\begin{notation}[$\J_n$, $\B_r(q)$, $\DD_{2r}(q)$] \label{not:matrices} Define the following matrices:
{\small \[
\J_n=
\begin{pmatrix}
0 &&&&(-1)^{n+1}\\
&&& \iddots&(-1)^{n}\\
&&1&\iddots&\\
&-1&-1&&\\
1&1&&&0
\end{pmatrix}_{n \times n}
\text{~~with } \J_1 = (1),
\quad 
\quad 
\B_r(q) = 
\begin{pmatrix}
q& 1 && 0\\
& q & \ddots &\\
&& \ddots & 1\\
0 &&& q
\end{pmatrix}_{r \times r}
\text{~~with }
\B_1(q) = (q),
\]}
{\small \[
\mathbb D_{2r}(q) = 
\begin{pmatrix}
0 & \I_r\\
\B_r(q) & 0
\end{pmatrix}_{2r \times 2r}
\quad \text{ with }\mathbb D_2(q) = 
\begin{pmatrix} 
0 & 1\\
q& 0
\end{pmatrix}.
\]}
\end{notation}

\begin{lemma} \cite[Theorem~2.1(a)]{HornSer2} \label{lem:congruent}
Each square matrix over $\kk$ is congruent to a direct sum, uniquely determined up to permutation of summands, of canonical matrices of the three types:
\begin{itemize}
\item $\B_r(0)$, 
\item $\J_n$, or 
\item $\mathbb D_{2r}(q), \text{ for } q \neq 0, (-1)^{r+1} \text{ where } q \text{ is determined up to replacement by } q^{-1}$. \qed
\end{itemize}
\end{lemma}

We use this result to study isomorphism classes of AS regular algebras of global dimension 2.

\begin{definition}[$A_{Sym}$, $A_{Skew}$, $A_{Jord}$, $A_{D_q}$] Consider the following special types of AS regular algebras of global dimension 2 associated to a nondegenerate bilinear form:
\begin{itemize}
\item {\it Symmetric type} corresponding to $\J_1^{\oplus n}$: 
$$A_{Sym}:= \kk\langle v_1, v_2, \dots, v_n \rangle/ (v_1^2 +v_2^2 + \cdots + v_n^2),$$ 
associated to the  symmetric nondegenerate bilinear form; 
\item {\it Skew-symmetric type} corresponding  to $\mathbb D_2(-1)^{\oplus \frac{n}{2}}$ for $n$ even: $$A_{Skew}:= \kk\langle v_1, v_2, \dots, v_n \rangle/ ([v_1,v_2] + \cdots + [v_{n-1},v_n]),$$ 
associated to the  skew-symmetric nondegenerate bilinear form; 
\item {\it Jordan type} corresponding  to $\J_n$: 
$$A_{Jord}:= \frac{\kk\langle v_1, v_2, \dots, v_n \rangle}{ 
((-1)^{n+1}v_1v_n + (-1)^nv_2v_{n-1} + (-1)^n v_2v_n + \cdots - v_{n-1}v_2 - v_{n-1}v_3 + v_nv_1 + v_nv_2)};$$ 
\item {\it Double quantum type}  corresponding  to $\mathbb D_{2r}(q)$ for $q \neq 0, (-1)^{r+1}$: 
$$A_{D_q}:= \frac{\kk\langle v_1, v_2, \dots,  v_{2r} \rangle}{ 
\left(
\begin{array}{c}
v_1v_{r+1}+q v_{r+1}v_1+v_2v_{r+2}+q v_{r+2}v_2+\cdots+v_{r}v_{2r}+q v_{2r}v_{r}\\
+v_{r+1}v_{2}+v_{r+2}v_3+\cdots+v_{2r-1}v_{r}
\end{array}
\right)}.
$$ 
\end{itemize}
\end{definition}
\medskip

Now one can describe explicitly the AS regular algebras $A(n, \E)$ of global dimension~2, up to isomorphism; we do so for $n=2, 3$ below.

\begin{corollary}  \label{cor:AS}
\textnormal{(a)} An AS regular algebra $A(2, \E)$ is isomorphic to either $A_{Sym}$, the Jordan plane $A_{Jord}=\kk\langle v_1, v_2 \rangle/ (v_2v_1-v_1v_2+v_2^2)$, or one of the quantum planes $A_{D_q}=\kk\langle v_1, v_2 \rangle/ (v_1v_2+qv_2v_1)$ for $q \in \kk \setminus \{0, 1\}$.

\textnormal{(b)} An AS regular algebra $A(3, \E)$ is isomorphic to either:
\begin{itemize}
\item $A_{Sym}= \kk\langle v_1, v_2, v_3 \rangle/ (v_1^2+v_2^2+v_3^2)$;
\item $A_{Jord}=\kk\langle v_1, v_2, v_3 \rangle/ (v_1v_3-v_2^2-v_2v_3+v_3v_1+v_3v_2)$;  
\item $\kk\langle v_1, v_2, v_3 \rangle/ (v_1^2-v_2v_3+v_3v_2+v_3^2)$; or 
\item  $\kk\langle v_1, v_2, v_3 \rangle/ (v_1^2 + v_2v_3+ qv_3v_2)$ for $q \in \kk \setminus \{0, 1\}$.
\end{itemize}
\end{corollary}

\begin{proof}
(a) This is well-known, but it also follows immediately from Lemma~\ref{lem:congruent}. Namely, the isomorphism classes of AS regular algebras $A(2, \E)$ are represented by $\J_1 \oplus \J_1$, $\J_2$, and $\mathbb D_2(q)$ for $q \neq 0, 1$. Here,  
\begin{itemize}
\item $\J_1 \oplus \J_1$ represents the quantum plane $\kk\langle v_1, v_2 \rangle/ (v_1v_2+v_2v_1) \cong \kk\langle v_1, v_2 \rangle/ (v_1^2+v_2^2) = A_{Sym}$; 
\item $\J_2$ represents the Jordan plane $A_{Jord}$; and 
\item $\mathbb D_2(q)$  represents the quantum planes (of double quantum type) $A_{D_q} = \kk\langle v_1, v_2 \rangle/ (v_1v_2 + q v_2v_1)$ for $q \in \kk \setminus \{0, 1\}$. This includes $\mathbb D_2(-1)$ for the commutative polynomial algebra $A_{Skew}$ on two variables.
\end{itemize}

(b) The isomorphism classes of AS regular algebras $A(3, \E)$ are represented by $\J_1^{\oplus 3}$, $\J_3$, $\J_1 \oplus \J_2$, and $\J_1 \oplus \mathbb D_2(q)$ for $q \in \kk \setminus \{0, 1\}$. Here, 
\begin{itemize}
\item $\J_1^{\oplus 3}$ represents $A_{Sym}= \kk\langle v_1, v_2, v_3 \rangle/ (v_1^2+v_2^2+v_3^2)$;
\item $\J_3$ represents $A_{Jord} = \kk\langle v_1, v_2, v_3 \rangle/ (v_1v_3-v_2^2-v_2v_3+v_3v_1+v_3v_2)$; 
\item $\J_1 \oplus \J_2$ represents $\kk\langle v_1, v_2, v_3 \rangle/ (v_1^2-v_2v_3+v_3v_2+v_3^2)$;
\item $\J_1 \oplus \mathbb D_2(q)$ for $q \in \kk \setminus \{0, 1\}$ represents  $\kk\langle v_1, v_2, v_3 \rangle/ (v_1^2 + v_2v_3+ qv_3v_2)$.
\end{itemize}
\vspace{-.2in}

\end{proof}

In general, we have the following result.

\begin{lemma} \label{lem:AS(E)}
Let $\E$ be the matrix associated to an AS regular algebra $A(\E)$ of global dimension 2. Then, $\E$ is congruent to the direct sum of matrices $\J_{n}$ and $\mathbb D_{2r}(q)$ with $q\neq 0, (-1)^{r+1}$.
\end{lemma}

\begin{proof}
Apply Theorem~\ref{thm:Zhang} and Lemmas~\ref{lem:isom} and~\ref{lem:congruent}.
\end{proof}


\subsection{Nakayama automorphisms and skew Calabi-Yau algebras} \label{sec:Nak}
Further background for the material in this section can be found in \cite{RRZ} and~\cite{BrownZhang:Dualizing}. Consider the following definition. 

\begin{definition} \label{def:skewCY} Let $R$ be an algebra with finite injective dimension $d$ on both sides. Let $R^e:=R\otimes R^{op}$. \begin{enumerate}
\item $R$ is called {\it skew Calabi-Yau } if 
\begin{enumerate}
\item[(i)] $R$ is {\it homologically smooth of dimension $d$}, that is, $R$ has a minimal projective resolution in the category $R^e$-Mod of length $d$ such that each term in the projective resolution is finitely generated, and
\item[(ii)] $R$ is {\it rigid Gorenstein}, that is,  there is an
algebra automorphism $\mu$ of $R$,  such that
\begin{equation*}
\Ext^i_{R^e}(R,R^{e})\cong \begin{cases}
^{\mu} R^1& {\text{if}} \quad i=d,
\\ 0& {\text{if}} \quad i\neq d,\end{cases}
\end{equation*}
as $R$-bimodules.
\end{enumerate}
\item The automorphism $\mu_R:=\mu$ is called a
{\it Nakayama automorphism} of $R$.
\item The algebra $R$ is {\it Calabi-Yau} if $\mu$ is inner (or equivalently, if $\mu$ is the identity after changing the generator of the bimodule $^{\mu} R^1$).
\end{enumerate} 
\end{definition}

The definition of $\mu$ is motivated by the classical notion of the Nakayama automorphism of a Frobenius algebra. The Nakayama automorphism $\mu$ is also unique up to an inner automorphism of $R$. Further, if $R$ is connected graded, then the Nakayama automorphism can be chosen to be a graded algebra automorphism, and in this case, it is unique since $R$ has no non-trivial graded inner automorphisms.

\begin{definition}[$\xi_r$] \label{def:rNaka} Let $R$ be a connected
${\mathbb N}$-graded algebra. Take $r\in \kk^\times$.  Define a graded algebra automorphism
$\xi_r$ of $R$ by
$$\xi_r(x)=r^{\deg x} \; x$$
for all homogeneous elements $x\in R$. We say that $R$ is $r$-{\it Nakayama} if  the Nakayama automorphism of $R$ (if it exists) is given by $\xi_r$.
\end{definition}

We have that $R$ is skew Calabi-Yau if and only if $R$ is AS regular \cite[Lemma~1.2]{RRZ}. Moreover, we get that $R$ is Calabi-Yau if and only if $R$ is both $1$-Nakayama and AS regular.
Moreover, consider the preliminary results below.

\begin{lemma} \label{lem:NakAutom} We have the statements below for any AS regular algebra $A:=A(n,\E)$ of global dimension~2.
\begin{enumerate} 
\item $A(n,\E)$ is skew Calabi-Yau.
\item The Nakayama automorphism of $A(n,\E)$  is given by
$\mu_A(v_i) = - \sum_{j=1}^n (\E^{-1}\E^T)_{ij} v_j$. 
\item $A(n,\E)$  is $r$-Nakayama only when $r=\pm 1$.
\item  $A(n,\E)$  is Calabi-Yau \textnormal{(}or $1$-Nakayama\textnormal{)} if and only if $\E$ is skew-symmetric and $n$ is even.
\item $A(n,\E)$  is $(-1)$-Nakayama if and only if $\E$ is symmetric.
\end{enumerate}
\end{lemma}
\begin{proof}
(a) This follows from \cite[Lemma~1.2]{RRZ}.

(b) Note that ${}^{\mu}A^1 \cong {}^{1}A^{\mu^{-1}}$ as $A$-bimodules.
Now by \cite[Proposition~3.3]{HvOZ}, the Nakayama automorphism of $A$ is given by
 $\mu^{-1}(v_i) = -\sum_{j=1}^n v_j (\E^T\E^{-1})_{ji} = -\sum_{j=1}^n ((\E^{-1})^T\E)_{ij} v_j $. Then, we take the inverse of $\mu^{-1}$ to get our formula. 

(c) This follows from part (b) and Lemmas \ref{lem:isom} and~\ref{lem:AS(E)} by direct computation.

(d,e) Note that the property of being symmetric or skew-symmetric is invariant under the congruence of matrices. Hence, parts (d,e) follow from part (c) and Lemmas \ref{lem:isom} and~\ref{lem:AS(E)} by direct computation.
\end{proof}


\subsection{Hopf algebra coactions and homological codeterminant} \label{sec:Hopf}

We refer to \cite{Montgomery} and \cite{KKZ:Gorenstein} for further background in this section.  We say that a  Hopf algebra $H$ {\it
coacts} on an algebra $R$ if $R$ arises as an $H$-{\it comodule
algebra}; refer to \cite[Definition~4.1.2]{Montgomery}
for the definition of an $H$-comodule
algebra. 

It is sometimes useful to restrict ourselves to $H$-coactions that do not factor through coactions of `smaller' Hopf subalgebras of $H$. For this, we provide the definition of inner-faithfulness.

\begin{definition} \label{def:innerfaithful}
Let $N$ be a right $H$-comodule
with comodule structure map $\rho: N \rightarrow N \otimes H$. We say
that this coaction $\rho$ is {\it inner-faithful} if $\rho(N)\not\subset
N\otimes H'$ for any proper Hopf subalgebra $H'\subsetneq H$.
\end{definition}

Now we recall an important invariant of Hopf algebra coactions on an AS regular algebra: homological codeterminant. To do so, note that by \cite[Corollary D]{LPWZ:Koszul}, a connected graded algebra $R$ is AS regular if and only
if the Yoneda algebra $E(R):=\bigoplus_{i\geq 0} \Ext^i_R(_R\kk,_R\kk)$ of $R$ is Frobenius. 

\begin{definition}[${\sf D}$] \label{def:hdet}
Let $R$ be an AS regular algebra with Frobenius Yoneda algebra $E(R)$ as above.
Let $H$ be a Hopf algebra with bijective antipode $S$ coacting on
$R$ from the right. We get that $H$ coacts on $E(R)$
from the left by \cite[Remark 1.6(d)]{CWZ:Nakayama}. Suppose ${\mathfrak e}$ is a nonzero element in
$\Ext^d_R(_R\kk,_R\kk)$ where $d=\gldim R$.
The {\it homological codeterminant} of the
$H$-coaction on $R$ is defined to be an element ${\sf
D} \in H$  where $\rho({\mathfrak e})={\sf D}\otimes {\mathfrak e}$. Note that  ${\sf D}$ is a group-like element of $H$. We say the homological
codeterminant is {\it trivial} if ${\sf D}=1_H$.
\end{definition}

Note that the homological codeterminant  of a Hopf coaction on an AS regular algebra is also realized as the {\it quantum determinant} introduced in \cite{Manin:QGNCG}; see \cite[Remark~2.4]{CWZ:Nakayama} for more details. In any case, we can adapt the proof of  \cite[Theorem~2.1]{CKWZ} to obtain the following result (the finite GK-dimension hypothesis is not needed).

\begin{proposition} \label{prop:hcodet} 
Let $A:=\kk\langle v_1, \dots, v_n \rangle/(r)$ be an AS regular algebra of global dimension~2, that admits a right coaction $\rho$ of a Hopf algebra $H$, induced by $\rho: \kk\langle v_1, \dots, v_n \rangle \rightarrow \kk\langle v_1, \dots, v_n \rangle \otimes H$. Then, the homological codeterminant ${\sf D}$ of the $H$-coaction on $A$ is given by $\rho(r) = r \otimes {\sf D}^{-1}$. \qed \end{proposition}


\section{Universal quantum linear groups} \label{sec:quantumlinear}

In this section, we discuss several universal quantum linear groups that coact on $A:=A(\E)$, in particular $\mc{O}_A(GL)$ and $\mc{O}_A(SL)$ due to Manin [Section~\ref{subsec:OA}], and $\mc{G}(\E^{-1},\F)$ and $\mc{B}(\E^{-1})$ due to Mrozinski, and Dubois-Violette and Launer, respectively [Section~\ref{subsec:B(E)}].


\subsection{Manin's quantum groups} \label{subsec:OA}
In this subsection, we discuss Manin's quantum linear groups introduced in \cite{Manin:QGNCG}, applied to $A:=A(n,\E)$.
Recall that  $A$ is a quadratic algebra
generated by
$v_1,v_2,\dots,v_n$ in degree 1, subject to the relation
$$r:=\sum_{1 \leq i,j \leq n} e_{ij} v_i v_j=0,$$
 where $(e_{ij}) =\E \in GL_n(\kk)$.
Let $\mc{F}$ be a free algebra generated by $\{a_{ij}\}_{1\leq i,j\leq n}$.
Consider a bialgebra structure on $\mc{F}$  defined by
$\Delta(a_{ij})=\sum_{s=1}^n a_{is}\otimes a_{sj}$ and
$\ee(a_{ij})=\delta_{ij}$,
for all $1\leq i,j \leq n$. The free
algebra $A':=\kk\langle v_1,\cdots,v_n\rangle$ is a right $\mc{F}$-comodule
algebra with comodule structure map $\rho: A'\to A'\otimes \mc{F}$ determined by
$\rho(v_i)=\sum_{s=1}^n v_s\otimes a_{si}$
for $1\leq i \leq n$.

The Koszul dual of $A$, denoted by $A^!$, is an
algebra $\kk\langle V^*\rangle/(R^{\perp})$ where $R^{\perp}$ is the
subspace \linebreak
$R^{\perp}:=\{ u\in (V^*)^{\otimes 2}\mid \langle u, r
\rangle=0\}.$ By duality,
$R:=\kk r =\{ w\in V^{\otimes 2}\mid \langle R^{\perp}, w\rangle=0\}.$
Pick a basis for $R^{\perp}$, say $r'_{t}$ for $t =1,\cdots, n^2-1$,
and write
$$r'_{t}= \sum_{1 \leq i,j \leq n}d_{ij}^{(t)} v^*_i v^*_j$$
for all $1\leq t \leq n^2-1$. By the definition of $R^{\perp}$, we
have that
$
\sum_{1\le i,j\le n} d_{ij}^{(t)} e_{ij}=0
$
for all $t$. Consider the following result.

\begin{lemma} \cite[Lemmas 5.5 and 5.6]{Manin:QGNCG}\label{lem:quadraticfactor}
 Let $H$ be a bialgebra coacting on the free algebra
$A'= \kk \langle v_1, \dots, v_n \rangle$ with $\rho(v_i)=\sum_{s=1}^n
v_s \otimes b_{si}$ for $b_{si}\in H$. Then, the following are
equivalent.
\begin{enumerate}
\item
The coaction $\rho : A' \longrightarrow A' \otimes H$ satisfies $\rho \left( \kk r
\right) \subseteq \kk r \otimes H$, that is, 
$\sum_{1\le i,j,k,l\le n} e_{ij} d_{kl}^{(t)} b_{ki}b_{lj}=0$
for all $t$.

\item
The map $\rho$ induces naturally a coaction of $H$ on $A$ such that
$A$ is a right $H$-comodule algebra. \qed
\end{enumerate}
\end{lemma}

Now we can define quantum linear groups associated to $A$. Let $\mc{I}$ be the ideal of the free bialgebra $\mc{F}$ generated by elements 
$$\sum\limits_{1\le i,j,k,l\le n} e_{ij}d_{kl}^{(t)} a_{ki}a_{lj}$$ for all $t$.

\begin{definition}[$\mc{O}_A(M)$, $\mc{O}_A(GL)$] \cite[Section 5.3 and Chapter~7]{Manin:QGNCG} \label{def:OA(M)}
Consider the following terms.
\begin{enumerate}
\item The {\it quantum matrix
space associated to $A$} is defined to be ${\mathcal O}_A(M) =$$\mc{F}/\mc{I}$, which is
a bialgebra quotient of $\mc{F}$.
\item The {\it Hopf envelope} of a bialgebra $B$ is a Hopf algebra $\mc{H}(B)$ equipped with a bialgebra map\linebreak  $\phi: B\to \mc{H}(B)$, so that for any Hopf algebra $H'$, any  bialgebra map $\psi: B\to H'$ extends uniquely to a Hopf algebra map $\hat{\psi}: \mc{H}(B)\to H'$ with $\psi=\hat{\psi}\circ \phi$ as bialgebra maps. \item The {\it quantum general linear group associated
to} $A$ is defined to be the Hopf envelope of ${\mathcal O}_A(M)$,
and is denoted by ${\mathcal O}_{A}(GL)$. 
\end{enumerate}
\end{definition}

Moreover, we introduce the following quotient of ${\mathcal O}_{A}(GL)$, which plays a central role in this work.

\begin{definition}[$\mc{O}^c_A(GL)$] \label{def:OcA(GL)} 
The {\it quantum general linear group associated
to} $A$ {\it with central quantum determinant} is defined to be the Hopf quotient ${\mathcal O}^c_{A}(GL):=\mc{O}_A(GL)/\mc{I}_{{\sf D}}$, where $\mc{I}_{\sf D}$ is the Hopf ideal generated by $\{a{\sf D} - {\sf D}a ~|~ a \in \mc{O}_A(GL)\}$.  
\end{definition}

Abusing notation, we use $a_{ij}$ to denote the image of the generators $a_{ij} $ of
$\mathcal{O}_A(M)$ in ${\mathcal O}_{A}(GL)$ and in ${\mathcal O}^c_{A}(GL)$. We see that $\mc{O}_A(M)$ (resp., $\mc{O}_A(GL)$, $\mc{O}^c_A(GL)$) serves as the universal bialgebra (resp., Hopf algebras) that
coacts on $A$ as follows.

\begin{lemma} \label{lem:OA(M)}
Suppose that the right coaction of $\mc{O}_A(M)$ on $A$ is given by $\rho_{\mc{O}}: A \rightarrow A \otimes \mc{O}_A(M)$. Then, for any bialgebra $B$ that coacts on $A$
from the right, via $\rho_{B}:A \rightarrow A \otimes B$, there exists a unique bialgebra map $\gamma: \mc{O}_A(M) \rightarrow B$ such that $\rho_{B} = (Id \otimes \gamma) \circ \rho_{\mc{O}}$.
\end{lemma}

\begin{proof}
Suppose $\rho_{B}(v_i)=\sum_{s=1}^{n}v_s\otimes b_{si}$ for all $b_{si}\in B$. We define a bialgebra map $\gamma: \mc{O}_A(M)\to B$ by  $\gamma(a_{ij})=b_{ij}$. By Definition \ref{def:OA(M)} and Lemma \ref{lem:quadraticfactor}, we know the map $\gamma$ is well-defined. Moreover, we have $\rho_{B} = (Id \otimes \gamma) \circ \rho_{\mc{O}}$ by checking this on the generators. Finally, the uniqueness of $\gamma$ is clear.
\end{proof}

\begin{lemma} \cite[Section~7.5]{Manin:QGNCG} \cite[Lemma~2.7]{CWZ:Nakayama}  \label{lem:OA(GL)}
Suppose that the right coaction of $\mc{O}_A(GL)$ on $A$ is given by $\rho_{\mc{O}}: A \rightarrow A \otimes \mc{O}_A(GL)$. Then, for any Hopf algebra $H$ that coacts on $A$
from the right, via $\rho_{H}:A \rightarrow A \otimes H$, we have the following statements. 
\begin{enumerate}
\item There is a unique Hopf algebra map $\gamma: \mc{O}_A(GL) \rightarrow H$ such that $\rho_{H} = (Id \otimes \gamma) \circ \rho_{\mc{O}}$.
\item Write $\rho_{H}(v_i)=\sum_{s=1}^n v_s\otimes b_{si}$
for some $b_{si}\in H$. Then, the $H$-coaction on $A$ is inner-faithful if and
only if,  for all $1\leq i,j\leq n$, the map $\gamma: a_{ij}\to b_{ij}$
induces a surjective Hopf algebra map from ${\mathcal
O}_{A}(GL)$ to $H$. \qed
\end{enumerate}
\end{lemma}

The following result is clear from the lemma above.

\begin{lemma} \label{lem:OcA(GL)}
Suppose that the right coaction of $\mc{O}^c_A(GL)$ on $A$ is given by $\rho_{\mc{O}^c}: A \rightarrow A \otimes \mc{O}^c_A(GL)$. Then, for any Hopf algebra $H$ that coacts on $A$
from the right with central homological codeterminant, via $\rho_{H}:A \rightarrow A \otimes H$, we have the following statements. 
\begin{enumerate}
\item There is a unique Hopf algebra map $\gamma: \mc{O}^c_A(GL) \rightarrow H$ such that $\rho_{H} = (Id \otimes \gamma) \circ \rho_{\mc{O}^c}$.
\item Write $\rho_{H}(v_i)=\sum_{s=1}^n v_s\otimes b_{si}$
for some $b_{si}\in H$. Then, the $H$-coaction on $A$ is inner-faithful if and
only if,  for all $1\leq i,j\leq n$, the map $\gamma: a_{ij}\to b_{ij}$
induces a surjective Hopf algebra map from ${\mathcal
O}^c_{A}(GL)$ to $H$. \qed
\end{enumerate}
\end{lemma}

Now we introduce several other quantum groups
associated to $A$. 

\begin{definition}[${\mathcal O}_{A}(SL)$, $\mathcal{O}_{A}(GL/S^{2m})$, $\mathcal{O}_{A}(SL/S^{2m})$, $\mathcal{O}_{A}(GL/S^\infty)$, $\mathcal{O}_A(SL/S^\infty)$, ${\mathcal O}^c_{A}(\ast)$] \label{def:OA(SL)}
Let ${\sf D}$ be the homological codeterminant of the ${\mathcal O}_A(GL)$-coaction on $A$. Inside ${\mathcal O}_A(GL)$, let $\mc{L}_{2m}$ be the Hopf ideal generated by the elements $\{S^{2m}(a)-a~|~ a\in {\mathcal O}_{A}(GL)\}$ for any integer $m\ge 1$.
\begin{enumerate}
\item \cite[Section 8.5]{Manin:QGNCG}
The {\it quantum special linear group associated to} $A$, denoted by
${\mathcal O}_{A}(SL)$, is defined to be the Hopf algebra quotient
${\mathcal O}_{A}(GL)/({\sf D}-1_{{\mathcal O}_{A}(GL)})$.
\item
The {\it quantum $S^{2m}$-trivial
general linear group associated to} $A$, denoted by ${\mathcal
O}_{A}(GL/S^{2m})$, is defined to be the Hopf algebra quotient
${\mathcal O}_{A}(GL)/\mc{L}_{2m}$.
\item The {\it quantum $S^\infty$-trivial general
linear group associated to} $A$, denoted by ${\mathcal
O}_{A}(GL/S^\infty)$, is defined to be the projective limit of Hopf algebras  
$\varprojlim\limits_{m\in \mathbb Z_{>0}} \mathcal {O}_{A}(GL/S^{2m})$. 
\item Similarly, we can define {\it quantum $S^{2m}$-trivial special
linear group associated to} $A$, namely ${\mathcal
O}_{A}(SL/S^{2m})$ and the {\it quantum $S^\infty$-trivial special
linear group}, denoted as ${\mathcal
O}_{A}(SL/S^\infty)$, by replacing ${\mathcal
O}_{A}(GL)$ with ${\mathcal O}_{A}(SL)$ in parts (b,c).
\item Moreover, we have versions of the quantum groups above so that the coaction on $A$ has central homological codeterminant. Denote such a quantum group by replacing $\mathcal{O}_A(\ast)$ with $\mathcal{O}^c_A(\ast)$.
\end{enumerate}
\end{definition}

\begin{remark} \label{rem:OASinfty}
(1) The projective limits for $\mc{O}_A(GL/S^{\infty})$,  $\mc{O}_A^c(GL/S^{\infty})$, and $\mc{O}_A(SL/S^{\infty})$ are taken in the category of coalgebras over the direct set $\mathbb Z_{>0}$ where $n<m$ if $n\mid m$; see reference \cite[Theorem 1.7]{Agore:LimitsHopfalgebra}. By the universal property, there is a unique Hopf algebra map from $\mc{O}_A(GL)$ to $\mc{O}_A(GL/S^{\infty})$, which factors through all the projections $\mc{O}_A(GL)\twoheadrightarrow\mc{O}_A(GL/S^{2m})$ for $m\ge 1$. Similar statements hold for $\mc{O}_A^c(GL/S^{\infty})$ and $\mc{O}_A(SL/S^{\infty})$.  

(2) We will see later in Lemma \ref{L:quantumGrp}(c) and Corollary~\ref{C:SLGL}(c) that the finite dimensional generating subspace $\bigoplus_{1\le i,j\le n} \kk a_{ij}$ of $\mc{O}^{c}_A(GL)$ is invariant under the linear operator $S^{2m} - I$. Moreover, the chain of images
$$\cdots\supseteq \bigcap_{1\le i\le m}\text{im}(S^{2i} - I) \supseteq  \bigcap_{1\le i\le m+1}\text{im}(S^{2i} - I)  \supseteq \dots$$ stabilizes for $m\gg 0$ in $\mc{O}^{c}_A(GL)$. Hence, $\mc{O}^{c}_A(GL/S^{\infty})$ is equal to $\mc{O}^{c}_A(GL/S^{2m})$ for some $m \gg 0$. See Proposition \ref{prop:GLJord} and Proposition \ref{prop:GLDq}, for instance. Similar statements hold for $\mc{O}_A(SL/S^\infty)$. 
\end{remark}

Now we discuss the universal property of the quantum groups above. By Remark \ref{rem:OASinfty}(1), the $S^\infty$-trival quantum group ${\mathcal O}_{A}(GL/S^\infty)$  coacts on $A$ from the right via the unique Hopf algebra map $\mc{O}_A(GL)\to \mc{O}_A(GL/S^{\infty})$.  Abusing notation, we still consider $\{a_{ij}\}_{1\le i,j\le n}$ as elements of ${\mathcal O}_{A}(GL/S^\infty)$ via the unique map above. The  lemma below follows from the universal
property of ${\mathcal O}_{A}(GL)$ in Lemma~\ref{lem:OA(GL)}.

\begin{lemma} \label{lem:OA(SL)}
Suppose that the right coaction of $\mc{O}_A(GL/S^\infty)$ on $A$ is given by $\rho_{\mc{O}^\infty}: A \rightarrow A \otimes \mc{O}_A(GL/S^\infty)$. Then, for any Hopf algebra $H$, with antipode of finite order, coacting on $A$, we have the following statements.
\begin{enumerate}
\item[(a)] There is a unique Hopf algebra map $\gamma: \mc{O}_A(GL/S^\infty) \rightarrow H$ such that $\rho_{H} = (Id \otimes \gamma) \circ \rho_{\mc{O}^\infty}$.
\item[(b)] Write $\rho_{H}(v_i)=\sum_{s=1}^n v_s\otimes b_{si}$
for some $b_{si}\in H$. Then, the $H$-coaction on $A$ is inner-faithful if and
only if,  for all $1\leq i,j\leq n$, the map $\gamma: a_{ij}\to b_{ij}$
induces a surjective Hopf algebra map from ${\mathcal
O}_{A}(GL/S^\infty)$ to $H$.
\qed
\end{enumerate}
\end{lemma}

We have a similar result for Hopf algebras $H$, with antipodes of finite order, that coact on $A$ subject to the condition that the homological codeterminant is central (resp. trivial); in this case, replace ${\mathcal O}_A(GL)$ with ${\mathcal O}^c_A(GL)$ (resp., ${\mathcal O}_A(SL)$), and replace ${\mathcal O}_A(GL/S^\infty)$ with ${\mathcal O}^c_A(GL/S^\infty)$ (resp., ${\mathcal O}_A(SL/S^\infty)$) in the lemma above.


\subsection{Dubois-Violette and Launer's and Mrozinski's quantum groups}
\label{subsec:B(E)}

In this subsection, we consider Hopf algebras, defined by Dubois-Violette and Launer \cite{DVL} and Mrozinski \cite{Mrozinski}, which play the role of function algebras on the quantum (symmetry) group of a non-degenerate bilinear form. 

\begin{definition}[$\mc{B}(\E)$, $\mc{G}(\E,\F)$]  \label{def:B(E)}
Let $n\geq 2$ be an integer, and $\E, \F \in GL_n(\kk)$.
\begin{enumerate}
\item \cite{DVL} Let $\mc{B}(\E)$ be the Hopf algebra with generators $\A := (a_{ij})_{1 \leq i,j \leq n}$ satisfying the relations :
$$\A \E^{-1} \A^{T} \E ~=~ \I ~=~ \E^{-1} \A^{T} \E \A ,$$
with $\Delta(a_{ij}) = \sum_{s=1}^n a_{is} \otimes a_{sj}$, $\ee(a_{ij}) = \delta_{ij}$, for all $1 \leq i, j \leq n$,  and 
$S(\A) = \E^{-1} \A^T \E$.
\medskip

\item 
\cite{Mrozinski} Let $\mc{G}(\E,\F)$ be the Hopf algebra with generators $\A = (a_{ij})_{1 \leq i,j \leq n}$, ${\sf D}$, ${\sf D}^{-1}$ satisfying the relations:
$$\A \E^{-1} \A^{T} \E ~=~ {\sf D} \I ~=~ \F \A^{T} \F^{-1} \A, \quad  \quad {\sf D}{\sf D}^{-1} = {\sf D}^{-1}{\sf D} = 1,$$
with $\Delta(a_{ij}) = \sum_{s=1}^n a_{is} \otimes a_{sj}$, $\Delta({\sf D}^{\pm 1}) = {\sf D}^{\pm 1} \otimes {\sf D}^{\pm 1}$, $\ee(a_{ij}) = \delta_{ij}$, $\ee({\sf D}^{\pm 1}) =1$, for all $1 \leq i, j \leq n$,  and 
$S(\A) =  \E^{-1} \A^T \E ({\sf D}^{-1}\I) = ({\sf D}^{-1}\I) \F\A^T\F^{-1}$, ~$S({\sf D}^{\pm 1}) = {\sf D}^{\mp 1}$.
\end{enumerate}
\end{definition}

For part (b), we take $\F^{-1}, \E^{-1}$ to be the matrices $A,B$ in \cite{Mrozinski}, respectively. Note that we have a surjective Hopf algebra map  $\mc{G}(\E,\E^{-1})\twoheadrightarrow \mc{B}(\E)~~ \textnormal{(}\cong \mc{G}(\E,\E^{-1})/(\sf D-1) \textnormal{)}$.

\begin{remark} \label{rem:B(J1)}
By convention, we extend the definition of $\mc{B}(\E)$ to the case when $\E = \J_1$. Here, $\mc{B}(\J_1) = \kk[\mathbb{Z}/2\mathbb{Z}]$, as Hopf algebras.
\end{remark}

\begin{lemma}\cite{DVL} \cite[Proposition~2.3]{Bichon:B(E)}  \cite[Theorem~1.3]{Mrozinski} \label{lem:B(PTEP)}
Let $\E, \F, \mathbb{P} \in GL_n(\kk)$. Then,  $\mc{B}(\E) \cong \mc{B}(\mathbb{P}^T \E \mathbb{P})$, as Hopf algebras. Moreover, if $\E^T\F^T\E\F = \lambda\I$ for some $\lambda \in \kk^{\times}$, then $$\mc{G}(\F,\E) ~~\cong~~ \mc{G}(\mathbb{P}^T \F \mathbb{P}, \mathbb{P}^{-1}\E (\mathbb{P}^{-1})^T) ~~\cong~~ \mc{G}(\mathbb{P}^T \E^{-1} \mathbb{P}, \mathbb{P}^{-1}\F^{-1} (\mathbb{P}^{-1})^T),$$ as Hopf algebras. \qed
\end{lemma}

\begin{lemma}\label{L:CentralGL}We have the following statements for the Hopf algebra $\mc{G}(\E,\F)$.
\begin{enumerate}
\item If $\E \F = \lambda \I$ for some $\lambda \in \kk^{\times}$, then the generator $\sf D$ is central in $\mc{G}(\E,\F)$. 
\item If $\F^T \E= \lambda \I$ for some $\lambda \in \kk^{\times}$, then $\mc{G}(\E,\F)$ is involutory. 
\end{enumerate}
In particular, the Hopf algebra $\mc{G}(\E^{-1},\E^T)$ is involutory, and the generator ${\D}$ is central in the Hopf algebra $\mc{G}(\E^{-1},\E)$, for all $\E \in GL_n(\kk)$.
\end{lemma}

\begin{proof}
(a) It suffices to check that ${\sf D}$ commutes with the generators $\A$ of $\mc{G}(\E,\F)$. This follows as
$$\A (\D \I) ~=~ \A (\F \A^{T} \F^{-1} \A) ~=~ \A (\lambda \I  \E^{-1} \A^{T} \lambda^{-1} \I \E \A) ~=~ (A \E^{-1} \A^{T} \E) \A ~=~ (\D \I) \A.$$

(b) By the antipode axiom, we get that 
\[
\begin{array}{rll}
S^2(\A) &= S(\E^{-1} \A^T \E {\sf D}^{-1}) &= S( \E^{-1}  \A^T {\sf D}^{-1}\E)\\
&= \E^{-1} {\sf D}S(\A)^T  \E
&= \E^{-1}{\sf D} ((\F^{-1})^T {\sf D}^{-1}  \A   \F^T)\E ~~~~\\&=  (\E^{-1}(\F^{-1})^T)\A  (\F^T \E).
\end{array}
\]
Now apply $\F^T \E = \lambda \I$ for $\lambda \in \kk^{\times}$ to get that $S^2= Id$.
\end{proof}

With use of Proposition~\ref{prop:hcodet}, the following lemma shows that there is no abuse of the notation ${\sf D}$ between Definitions~\ref{def:hdet} and~\ref{def:B(E)}.

\begin{lemma} \label{lem:B(E)coact} Suppose $\E,\F\in GL_n(\kk)$. We have the following statements.
\begin{enumerate}
\item The algebra $A(n, \E)$ is a right $\mc{G}(\E^{-1},\F)$-comodule algebra via the coaction $\rho_{\mc{G}}(v_i) = \sum_{s=1}^n v_s \otimes a_{si}$ for all $1\le i\le n$, for any $\F$. We also have that 
$\rho_{\mc{G}} \textstyle \left(\sum_{1\le i,j\le n}e_{ij}v_iv_j\right)=\left(\sum_{1\le i,j\le n}e_{ij}v_iv_j\right)\otimes \sf D$.
\item The algebra $A(n, \E)$ is a right $\mc{B}(\E^{-1})$-comodule algebra via the coaction $\rho_{\mc{B}}(v_i) = \sum_{s=1}^n v_s \otimes a_{si}$ for all $1\le i\le n$. We have that $\rho_{\mc{B}}\left(\sum_{1\le i,j\le n}e_{ij}v_iv_j\right)=\left(\sum_{1\le i,j\le n}e_{ij}v_iv_j\right)\otimes \sf 1$.
\end{enumerate}
\end{lemma}

\begin{proof}
(a) It suffices to show that the relation $\sum_{1\le i,j\le n}e_{ij}v_iv_j$ is a one-dimensional $\mc{G}(\E^{-1},\F)$-comodule under the coaction $\rho_{\mc{G}}$. Direct computation shows that
\[
\begin{array}{lll}
\rho_{\mc{G}} \textstyle \left(\sum_{1\le i,j\le n}e_{ij}v_iv_j\right)
&=\textstyle \sum_{1\le i,j,k,l\le n}e_{ij}(v_kv_l)\otimes (a_{ki}a_{lj})
&=\textstyle \sum_{1\le k,l\le n}(v_kv_l)\otimes (\sum_{1\le i,j\le n}a_{ki}e_{ij}a_{jl}^T)\\
&=\textstyle \sum_{1\le k,l\le n}(v_kv_l)\otimes (\A\E\A^T)_{kl}
&=\textstyle \sum_{1\le k,l\le n } v_kv_l\otimes (\E{\sf D})_{kl}\\
&=\textstyle \left(\sum_{1\le k,l\le n } e_{kl}v_kv_l\right)\otimes \sf D.
\end{array}
\]

(b) For the statement about $\mc{B}(\E^{-1})$, take ${\sf D} =1$ (and $\F = \E$) in the argument above.
\end{proof}


\subsection{Relationship between quantum groups}

Now we establish a connection between the Manin's, Dubois-Violette and Launer's, and Mrozinski's quantum groups as follows. 

\begin{lemma}\label{L:quantumGrp} 
Let $H$ be a Hopf algebra that coacts on $A(\E)$ with homological codeterminant ${\sf D}^{-1}$.  Consider the matrix $\B:=(b_{ij})$ such that $\rho(v_i)=\sum_{1\le s\le n} v_s\otimes b_{si}$ for $b_{si} \in H$. Then, 
\begin{enumerate}
\item $\B\E\B^T\E^{-1}={\sf D}\I=\E\B^T\E^{-1}{\sf D}^{-1} \B {\sf D}$;
\item $\B\E\B^T\E^{-1}={\sf D}\I= \E\B^T\E^{-1} \B$ \quad if ${\sf D}$ is central; 
\item $S^2(\B)={\sf D}\M \B\M^{-1}{\sf D}^{-1}$, \quad where $\M=(\mu_{A(\E)}|_{A_1})^T=-\E(\E^{-1})^T$; and
\item $\B\E\B^T\E^{-1}={\sf D}\I = \E^T\B^T(\E^{-1})^T \B$ \quad if $H$ is involutory. 
\end{enumerate}
\end{lemma}

\begin{proof}
(a) We have that the relation $\sum_{1\le i,j\le n}e_{ij}v_iv_j$ of $A(\E)$ generates a one-dimensional right $H$-comodule.  By Proposition~\ref{prop:hcodet}, we have that
$$\rho \textstyle \left(\sum_{1\le i,j\le n}e_{ij}v_iv_j\right) = 
\textstyle \left(\sum_{1\le i,j\le n}e_{ij}v_iv_j\right) \otimes {\sf D} =\textstyle \left(\sum_{1\le i,j\le n}v_iv_j\right) \otimes (\E{\sf D})_{ij}.$$
We also get that 
\[
\begin{array}{ll}
\rho \textstyle \left(\sum_{1\le i,j\le n}e_{ij}v_iv_j\right)
&=\textstyle \sum_{1\le i,j,k,l\le n}e_{ij}(v_kv_l)\otimes (b_{ki}b_{lj})\\
&=\textstyle \sum_{1\le k,l\le n}(v_kv_l)\otimes (\sum_{1\le i,j\le n}b_{ki}e_{ij}b_{jl}^T)\\
&=\textstyle \sum_{1\le k,l\le n}(v_kv_l)\otimes (\B\E\B^T)_{kl}.
\end{array}
\]
Hence, we have that $\B\E\B^T=\E {\sf D}$, which implies that $\B\E\B^T\E^{-1}={\sf D}\I$. So, half of part (a) holds.

Since $\B S(\B)=\I$, for $\Delta(b_{ij})=\sum_{s=1}^{n} b_{is}\otimes b_{sj}$, we get 
$$S(\B)=\E\B^T\E^{-1}{\sf D}^{-1}\I. $$ 
Moreover $S(\B)\B = \I$, so we have that $\E\B^T\E^{-1}{\sf D}^{-1}\B{\sf D}={\sf D}\I$. Thus, part (a) holds.

(b) This follows immediately from (a).

(c) This follows from \cite[Theorem~0.1]{CWZ:Nakayama} and Lemma~\ref{lem:NakAutom}(b) (or by direct computation).

(d) We get from part (a) and $S^2(\B) = \B$ that
$${\sf D}\I = \E\B^T\E^{-1}[{\sf D}^{-1} \B {\sf D}] = \E\B^T\E^{-1}[\E(\E^{-1})^T\B\E^T\E^{-1}] = \E\B^T(\E^{-1})^T\B\E^T\E^{-1}, $$
which yields $\E^T\B^T(\E^{-1})^T \B = {\sf D}\I$. The other relation follows from part (a).
\end{proof}

Now to get the connection between quantum groups discussed in the previous subsections, consider the following definition and corollary to the lemma above.

\begin{definition}[$\mathcal{M}(\E)$]
Let $n\geq 2$ be an integer, and $\E \in GL_n(\kk)$. Consider the bialgebra $\mathcal{M}(\E)$ with generators $\A:=(a_{ij})_{1\le i,j\le n}$ and ${\sf D}$ satisfying the relation
$\A \E^{-1} \A^{T} \E = {\sf D} \I$, with $\Delta(a_{ij}) = \sum_{s=1}^n a_{is} \otimes a_{sj}$, $\Delta({\sf D}) = {\sf D} \otimes {\sf D}$, $\ee(a_{ij}) = \delta_{ij}$, and $\ee({\sf D}) =1$.
\end{definition}

\begin{corollary}\label{C:SLGL} We have the following isomorphisms of various quantum linear \textnormal{(}semi\textnormal{)}groups.
\begin{enumerate}
\item $\mc{O}_{A(\E)}(M) \cong \mc{M}(\E^{-1})$, as bialgebras; and
\item $\mc{O}_{A(\E)}(SL)\cong \mc{B}(\E^{-1})$,
\item $\mc{O}^{c}_{A(\E)}(GL)\cong \mc{G}(\E^{-1},\E)$,
\item $\mc{O}_{A(\E)}(GL/S^2)\cong \mc{G}(\E^{-1},\E^T)$, as Hopf algebras.
\end{enumerate}
\end{corollary}

\begin{proof}
We will show that $\mc{O}^{c}_{A(\E)}(GL)\cong  \mc{G}(\E^{-1},\E)$ as Hopf algebras. Parts (a),(b) and (d) follow in the same fashion. By Lemmas~\ref{lem:B(E)coact} and~\ref{L:CentralGL}, there exists a coaction of $ \mathcal{G}^c:=\mc{G}(\E^{-1},\E)$ on $A(\E)$
 with central homological codeterminant ${\sf D}^{-1}$. By Lemma~\ref{lem:OcA(GL)}, we then get a unique Hopf algebra map $\gamma: \mc{O}^{c}:=\mc{O}^{c}_A(GL) \rightarrow  \mc{G}^c$ such that $\rho_{\mc{G}^c} = (Id \otimes \gamma) \circ \rho_{\mc{O}^{c}}.$ On the other  hand, by Definition~\ref{def:OcA(GL)} and Proposition~\ref{prop:hcodet}, there exists a coaction of $\mathcal{O}^{c}_A(GL)$ on $A(\E)$ with central homological codeterminant 
 ${\sf D}_\mc{O}^{-1}$. Now by Lemma~\ref{L:quantumGrp}(b), there exists a unique Hopf algebra map 
 $\eta:  \mc{G}^c \rightarrow \mc{O}^c$ such that $\eta({\sf D}^{-1})={\sf D}_\mc{O}^{-1}$ and $\rho_{\mc{O}^{c}} = (Id \otimes \eta) \circ \rho_{\mc{G}^c}$. Therefore, by uniqueness, $\gamma$ and $\eta$ are mutually inverse, and $\mc{G}^c$ and $\mc{O}^{c}$ are isomorphic as Hopf algebras.
\end{proof}

It is clear that there exists a surjection of Hopf algebras $\mc{O}^{c}_{A(\E)}(GL) \twoheadrightarrow \mc{O}_{A(\E)}(SL)$. Now we provide some examples of the universal quantum linear groups above. Some of these computations overlap with those presented in \cite[Section~5]{CWZ:Nakayama}.

\begin{example} \label{ex:B(Aq)}
We compute the  quantum linear groups  associated to the quantum plane $A_q:=A(2, \DD_2(q))$. We will see that $\mc{O}^{c}_{A_q}(GL)\cong \mc{O}_{-q}(GL_2)$, $\mc{O}_{A_q}(SL)\cong \mc{O}_{-q}(SL_2)$, and Takeuchi's two-parameter quantum group $\mc{O}_{A_q}(GL/S^2)\cong \mc{O}_{-q,-q^{-1}}(GL_2)$. The Hopf algebras on the right hand side are presented in \cite[Section~I.1]{book:BrownGoodearl} and \cite{Takeuchi:two-parameter}.
\begin{enumerate}
\item We have generators $a,b,c,d$ and group-like element ${\sf D}$ for $\mc{O}_{A_q}(M)$, and the following relations: 
$$ab=-qba, \quad cd=-qdc, \quad  ad+qbc= da + q^{-1}cb = {\sf D},$$
with coalgebra structure given by
$\Delta(a) = a\otimes a + b \otimes c$, 
$\Delta(b) = a\otimes b + b \otimes d$, 
$\Delta(c) = c\otimes a + d \otimes c$,
$\Delta(d) = c\otimes b + d \otimes d$,
$\ee(a)= \ee(d) =1$, and  $\ee(b)= \ee(c) =0$.
\item We pick up an extra generator ${\sf D}^{-1}$ for $\mc{O}^{c}_{A_q}(GL)$. The relations of  $\mc{O}^{c}_{A_q}(GL)$ include the relations of $\mc{O}_{A_q}(M)$ along with 
$$bd=-qdb, \quad ac=-qca, \quad  ad+qcb= da + q^{-1}bc = {\sf D}, \quad {\sf D}{\sf D}^{-1} = {\sf D}^{-1}{\sf D} = 1;$$
$\mc{O}^{c}_{A_q}(GL)$ has the same coalgebra structure as $\mc{O}_{A_q}(M)$. The antipode of $\mc{O}^{c}_{A_q}(GL)$ is given by
$$S(a) =  d{\sf D}^{-1}, \quad S(b) = q^{-1} b{\sf D}^{-1}, \quad S(c) =  qc{\sf D}^{-1}, \quad S(d) =  a{\sf D}^{-1}, \quad S({\sf D}^{\pm 1}) = {\sf D}^{\mp 1}.$$

\item In $\mc{O}^{c}_{A_q}(GL/S^{2m})$, we get the additional relations 
$$a{\sf D}^m = {\sf D}^ma, \quad q^{-2m}b{\sf D}^m = {\sf D}^mb, \quad q^{2m}c{\sf D}^m = {\sf D}^m c, \quad d{\sf D}^m = {\sf D}^md.$$
Suppose that $q^{2m}\neq 1$. Since ${\sf D}$ is central in $\mc{O}^{c}_{A_q}(GL)$, we have $b=c=0$. So, in this case, $\mc{O}^{c}_{A_q}(GL/S^{2m})$ is generated by group-like elements $a, d, {\sf D}^{\pm 1}$ where $ad=da=\mathsf D$. So, $\mc{O}^{c}_{A_q}(GL/S^{2m})\cong \mathbb \kk[\mathbb Z\times \mathbb Z]$. Else if $q^{2m}=1$, we have that $\mc{O}^{c}_{A_q}(GL/S^{2m})=\mc{O}^{c}_{A_q}(GL)$.

\item Hence, if $q$  is not a root of unity, then $\mc{O}^{c}_{A_q}(GL/S^{\infty})$ = $\mc{O}^{c}_{A_q}(GL/S^2)\cong \kk[\mathbb Z\times \mathbb Z]$ as Hopf algebras, by part (c). Otherwise, if $q$  is a root of unity, then $\mc{O}^{c}_{A_q}(GL/S^{\infty})$ = $\mc{O}^{c}_{A_q}(GL)$. 
\item Take ${\sf D} =1$ in (b), (c), (d) to get  $\mc{O}_{A_q}(SL)$,  $\mc{O}_{A_q}(SL/S^{2m})$, and $\mc{O}_{A_q}(SL/S^{\infty})$, respectively.

\item For arbitrary homological codeterminant, we have that  $\mc{O}_{A_q}(GL/S^2)$ is generated $a,b,c,d$ and group-like elements ${\sf D}^{\pm 1}$ with the following relations: ${\sf D} {\sf D}^{-1} = {\sf D^{-1}} {\sf D}=1$ and
\[
\begin{array}{llll}
&\quad \quad \quad ab=-qba, &\quad cd=-qdc, &\quad  ad+qbc= da + q^{-1}cb = {\sf D},\\
&\quad \quad \quad bd=-q^{-1}db, &\quad ac=-q^{-1}ca, &\quad  ad+q^{-1}cb= da + qbc = {\sf D}.
\end{array}
\]
The coalgebra structure and antipode is the same as given in parts (a) and (b).
\end{enumerate}
\end{example}

\begin{example}\label{ex:B(AJ)}
We compute the  quantum linear groups  associated to the quantum plane $A_J:=A(2, \J_2)$. 

\begin{enumerate}
\item We have generators $a,b,c,d$ and group-like element ${\sf D}$ for $\mc{O}_{A_J}(M)$, and the following relations: 
$$\quad \quad \quad [a,b]=b^2, \quad [c,d]=cb-da-db+d^2, \quad -ab+ad+ba+b^2-bc-bd~ = ~-cb+da+db ~=~{\sf D},$$
with coalgebra structure given by $\Delta(a) = a\otimes a + b \otimes c$, $\Delta(b) = a\otimes b + b \otimes d$,
$\Delta(c) = c\otimes a + d \otimes c$, 
$\Delta(d) = c\otimes b + d \otimes d$,
$\ee(a)= \ee(d) =1$,  and $\ee(b)= \ee(c) =0.$
\item We pick up an extra generator ${\sf D}^{-1}$ for $\mc{O}^{c}_{A_J}(GL)$. The relations of  $\mc{O}^{c}_{A_J}(GL)$ include the relations of $\mc{O}_{A_J}(M)$ along with ${\sf D}{\sf D}^{-1} = {\sf D}^{-1}{\sf D} = 1$, and
$$\quad \quad \quad [d,b]=b^2, \quad [c,a]=a^2+ba+bc-da, \quad -ba-bc+da ~=~ ab+ad+b^2+bd-cb-db ~=~  {\sf D};$$ 
$\mc{O}^{c}_{A_J}(GL)$ has the same coalgebra structure as $\mc{O}_{A_J}(M)$. The antipode of $\mc{O}^{c}_{A_J}(GL)$ is given by
$$ \quad \quad \quad
 S(a) = (d-b){\sf D}^{-1}, ~~ S(b) = -b{\sf D}^{-1},
~~S(c) = (a+b-c-d){\sf D}^{-1}, ~~ S(d) = (a+b){\sf D}^{-1}, ~~S({\sf D}^{\pm 1}) = {\sf D}^{\mp 1}.
$$
\item In $\mc{O}^{c}_{A_J}(GL/S^{2m})$, we get the additional relations 
\begin{gather*}
(a+2mb){\sf D}^m = {\sf D}^ma, \quad b{\sf D}^m = {\sf D}^mb,\\
 \quad (-2ma-4m^2b+c+2md){\sf D}^m = {\sf D}^m c, \quad (d-2mb){\sf D}^m = {\sf D}^md.
 \end{gather*}
Since ${\sf D}$ is central in $\mc{O}^{c}_{A_J}(GL)$, we have the additional relations: $a=d,~b=0$.
So, for all $m \geq 1$:
$\mc{O}^{c}_{A_J}(GL/S^{2m})$  is generated by $a, c$ and group-like elements ${\sf D}^{\pm 1}$,
 with  $[c,a]=0$, ${\sf D}=a^2$, and $\Delta(a) = a \otimes a, ~~\Delta(c) = a \otimes c+ c \otimes a$.
\item So, $\mc{O}^{c}_{A_J}(GL/S^{\infty})$ = $\mc{O}^{c}_{A_J}(GL/S^2)$ as Hopf algebras, by part (c).
\item  Take ${\sf D} =1$ in (b), (c), (d) to get  $\mc{O}_{A_J}(SL)$,  $\mc{O}_{A_J}(SL/S^{2m})$, and $\mc{O}_{A_J}(SL/S^{\infty})$, respectively.
\item For arbitrary homological codeterminant, we have that  $\mc{O}_{A_J}(GL/S^2)$ is generated $a,b,c,d$ and group-like elements ${\sf D}^{\pm 1}$ with the following relations: ${\sf D} {\sf D}^{-1} = {\sf D^{-1}} {\sf D} =1$ and
\[
\begin{array}{llll}
&\quad \quad \quad [a,b]=b^2, & [c,d]=cb-da-db+d^2, &-ab+ad+ba+b^2-bc-bd~ = ~-cb+da+db ~=~{\sf D},\\
&\quad \quad \quad [d,b]=-b^2, &[c,a]=-a^2+ba-bc+da, &ba-bc+da ~=~ -ab+ad+b^2-bd-cb+db ~=~  {\sf D}.
\end{array}
\]
The coalgebra structure and antipode is the same as given in parts (a) and (b).
\end{enumerate}
\end{example}

Now we present an example of an infinite dimensional, noncommutative, noncocommutative, cosemisimple Hopf algebra that coacts on the polynomial ring $\kk[u,v]$ inner-faithfully.  This shows that \cite[Theorem~1.3]{PavelWalton} cannot be extended to the setting of infinite dimensional, cosemisimple Hopf coactions on commutative domains, in general. This example grew out of conversations between the first author, Pavel Etingof, and Debashish Goswami; we thank Etingof and Goswami for allowing us to include this example here.

\begin{example}\label{Ex:NoncomAction}
Take $q\in \kk^\times$ a non-root of unity, and take $\E=\DD_2(-1)$ and $\F=\DD_2(-q)$. By Definition~\ref{def:B(E)}(b), $\mc{G}(\E^{-1},\F^{-1})$ is generated by $a,b,c,d$ and ${\sf D},{\sf D}^{-1}$ subject to relations: ${\sf D} {\sf D}^{-1} = {\sf D^{-1}} {\sf D}=1$, 
\[
\begin{array}{llllll}
 ba=ab, &\quad cd=dc, &\quad db=q^{-1}bd, &\quad ac=qca,&\quad  ad-bc= da-cb =da-q^{-1}bc= ad-qcb= {\sf D}.
\end{array}
\]
It is easy to see that this algebra is a localization of an iterated Ore extension at the denominator set formed by the normal element ${\sf D}$. Thus, this algebra is infinite dimensional and noncommutative.

The coalgebra structure of $\mc{G}(\E^{-1},\F^{-1})$ given by $\Delta(a) = a\otimes a + b \otimes c$, $\Delta(b) = a\otimes b + b \otimes d$,
$\Delta(c) = c\otimes a + d \otimes c$, 
$\Delta(d) = c\otimes b + d \otimes d$, $\Delta({\sf D}^{\pm 1})={\sf D}^{\pm 1}\otimes {\sf D}^{\pm 1}$, and
$\ee(a)= \ee(d)=\ee({\sf D})=1$,  and $\ee(b)= \ee(c) =0$. So,  $\mc{G}(\E^{-1},\F^{-1})$ is noncocommutative. The antipode of $\mc{G}(\E^{-1},\F^{-1})$ is given by
$S(a)=d{\sf D}^{-1},\ S(b)=-b{\sf D}^{-1},\ S(c)=-c{\sf D}^{-1},\ S(d)=a{\sf D}^{-1},\ S({\sf D}^{\pm 1})={\sf D}^{\mp 1}.$

According to Lemma \ref{lem:B(E)coact}(a), we know that $\mc{G}(\E^{-1},\F^{-1})$ coacts on the polynomial algebra $A(\E)\cong \kk[u,v]$ via $\rho(u)=u\otimes a+v\otimes c, ~\rho(v)=u\otimes b+v\otimes d$. This coaction is inner-faithful due to the universal property of $\mc{G}$. Moreover, since $(\E^{-1})^T(\F^{-1})^T\E^{-1}\F^{-1}=q^{-1}\I$, we apply \cite[Theorem 1.1]{Mrozinski} to conclude that  $\mc{G}(\E^{-1},\F^{-1})$ is cosemisimple. Thus, we have the desired Hopf coaction.
\end{example} 

In the example above, note that the noncommutative Hopf algebra $\mc{G}(\E^{-1},\F^{-1})$, with $\E=\DD_2(-1)$ and $\F=\DD_2(-q)$, is not involutory. On the other hand, the universal involutory quantum linear group associated to the polynomial algebra $A(2,\E)$ is always commutative by Example~\ref{ex:B(Aq)}(f). This prompts the following question due to Julien Bichon.

\begin{question}[J. Bichon]
Can an infinite dimensional, noncommutative, cosemisimple, involutory Hopf algebra  coact on a commutative domain inner-faithfully? 
\end{question}


\section{Homological properties of $\mc{O}_{A(\E)}(SL)$} \label{sec:hom}

In this section, we verify several homological properties of the quantum groups $\mc{O}_{A(\E)}(SL) \cong \mc{B}(\E^{-1})$ (see Corollary~\ref{C:SLGL}). Namely, we prove Theorem~\ref{thm:homintro}. The results below follow essentially from work of Bichon~\cite{Bichon}.

First, we need the preliminary result, which follows from a routine computation. Note that for a Hopf algebra $(H,m,u,\Delta,\ee,S)$ with bijective antipode $S$,  the {\it opposite Hopf algebra} $H^{op}$ is given by $(H,m^{op},u,\Delta,\ee,S^{-1})$.

\begin{lemma} \label{lem:B(E)op}
We have that $\mc{B}(\E)^{op}\cong \mc{B}(\E^T)$ as Hopf algebras. \qed
\end{lemma}

The result above is expected as $A(\E)^{op}\cong A(\E^T)$ as $\kk$-algebras.
Now we establish Theorem~\ref{thm:homintro} via the results below.

\begin{proposition}\label{prop:AS}
The quantum group $\mc{B}:=\mc{B}(\E)$ is Artin-Schelter regular of global dimension 3.
\end{proposition}

\begin{proof}
By forgetting the right $\mc{B}$-coaction of the Yetter-Drinfeld resolution of the counit of $\mc{B}$ presented in \cite[Theorem~5.1]{Bichon}, we obtain a length 3 resolution  of the trivial module $\kk_{\mc{B}}$. Now we have that  l.gl.dim($\mc{B}$) = r.gl.dim($\mc{B}$)$ ~\leq 3$ by \cite[page 37]{LL} and Lemma~\ref{lem:B(E)op}. By \cite[Theorem~6.1]{Bichon}, we get that $\mc{B}$ is homologically smooth of dimension 3. So, by \cite[Lemma~5.2(a)]{BrownZhang:Dualizing}, we also get that gl.dim($\mc{B}$)$ ~\geq 3$. Thus, gl.dim($\mc{B}$)$ ~=3$.

Now, with Lemma~\ref{lem:B(E)op}, it suffices to show that $\mc{B}$ is right AS Gorenstein. By \cite[Proposition~6.2]{Bichon}, we have that $\Ext_{\mc{B}}^i(\kk_{\mc{B}}, ~M_{\mc{B}}) = \text{Tor}_{\mc{B}}^{3-i}(\kk_{\mc{B}},~ _{\theta}M)$,  for $i =0,1,2,3$. Here, $\theta$ is the algebra anti-automorphism of $\mc{B}$ defined by $\theta(\A) = S(\A)\E^{-1}\E^T \E^{-1} \E^T$, and $_{\theta} M$ has the left $\mc{B}$-module structure given by $b \cdot m = m \cdot \theta(b)$. Take $M_{\mc{B}}$ to be $\mc{B}_{\mc{B}}$. Then, it is easy to see that $_{\theta} M \cong ~_{\mc{B}}\mc{B}$. Since $_{\mc{B}}\mc{B}$ is projective, we have that $\text{Tor}_{\mc{B}}^i(\kk_{\mc{B}},~_{\mc{B}}\mc{B}) = \kk \otimes_{\mc{B}} \mc{B} =\kk$ if $i=0$, and 0 otherwise. Thus, $\Ext_{\mc{B}}^3(\kk_{\mc{B}}, ~\mc{B}_{\mc{B}}) = \kk$, and  $\Ext_{\mc{B}}^i(\kk_{\mc{B}}, ~\mc{B}_{\mc{B}}) = 0$ for $i \neq 3$, as desired.
\end{proof}

\begin{proposition}\label{prop:skewCY}
The quantum group $\mc{B}(\E)$ is skew Calabi-Yau \textnormal{(}homologically smooth of dimension 3\textnormal{)}.
\end{proposition}

\begin{proof}
By \cite[Theorem~6.1]{Bichon}, we have that $\mc{B}$ is homologically smooth of dimension 3. Since $\mc{B}$ has a bijective antipode, the rigid Gorenstein property follows from the AS Gorenstein property established in Proposition~\ref{prop:AS}. Namely apply \cite[Lemmas~2.2(c), 2.4(b), 4.5]{BrownZhang:Dualizing} as follows; all of these results  do not require $\mc{B}$ to be Noetherian. We first get, for all $i$, that $\Ext^i_{\mc{B}^e}(\mc{B},\mc{B}^e) = \Ext^i_{\mc{B}}(\kk, L(\mc{B}^e))$, where $L(\mc{B}^e)$ is a free left $\mc{B}$-module defined by $b \cdot m = \sum b_1 m S(b_2)$ for all $m \in B^e$. So, we get that $\Ext^i_{\mc{B}^e}(\mc{B},\mc{B}^e) = \Ext^i_{\mc{B}}(\kk, \mc{B}) \otimes_{\mc{B}} L(\mc{B}^e)$, for all $i$. By the AS Gorenstein property, we now have $\Ext^3_{\mc{B}^e}(\mc{B},\mc{B}^e) = \kk \otimes_{\mc{B}} L(\mc{B}^e)$ and $\Ext^i_{\mc{B}^e}(\mc{B},\mc{B}^e) = 0$, for $i \neq 3$. Since $\kk \otimes_{\mc{B}} L(\mc{B}^e) \cong {}^\mu \mc{B}^{1}$ as $\mc{B}$-bimodules, for some algebra automorphism $\mu$ of $\mc{B}$, we are done.
\end{proof}

Now we recall from \cite{Bichon} the computation of the Nakayama automorphism of $\mc{B}:=\mc{B}(\E)$.  

\begin{lemma} \cite[Corollary~6.3]{Bichon} \label{L:NAutoB} 
The Nakayama automorphism $\mu_{\mc{B}}$ of $\mc{B}(\E)$ is given by $$\mu_{\mc{B}}(\A)=\E^{-1}\E^T\A\E^{-1}\E^T.$$

\vspace{-.25in} \qed 
\end{lemma}

\begin{corollary}\label{cor:CY} 
If $\E$ is symmetric or skew-symmetric, then $\mc{B}(\E)$ is both Calabi-Yau and involutory. 
\end{corollary}
\begin{proof}
The antipode of $\mc{B}(\E)$ is given by $S(\A) = \E^{-1}\A^T \E$. So, $S^2(\A) = S(\E^{-1} \A^T \E) = \E^{-1} S(\A)^T \E = \X \A \X^{-1}$, for $\X=\E^{-1}\E^T$. Since $\E$ is symmetric or skew-symmetric, we get that $\X = \pm \I$ and  $\mc{B}(\E)$ is involutory. 
 Moreover, we have that $\mc{B}(\E)$ is Calabi-Yau by Lemma \ref{L:NAutoB}.
\end{proof}


\section{Ring-theoretic properties of $\mc{O}_{A(\E)}(GL/S^2)$, $\mc{O}^c_{A(\E)}(GL)$, and $\mc{O}_{A(\E)}(SL)$} \label{sec:ring}

In this section, we establish ring-theoretic properties of the variants of the quantum general linear groups $\mc{O}^c_{A(\E)}(GL) \cong \mc{G}(\E^{-1},\E)$, $\mc{O}_{A(\E)}(GL/S^2) \cong \mc{G}(\E^{-1},\E^T)$, and quantum special linear groups $\mc{O}_{A(\E)}(SL) \cong \mc{B}(\E^{-1})$ (see Corollary~\ref{C:SLGL}). Namely, we verify Theorem~\ref{thm:ringintro}; see Theorem~\ref{thm:ring} below. To begin, consider the following notation.

\begin{notation}[$\T_{i,j}$, $\LL$, $\U$] \label{not:e,g,h} Take $\T_{i,j}$ to be the matrix in $M_n(\kk)$ with $1$ in the $(i,j)$-th entry and $0$ elsewhere. We denote by $\LL=\sum_{2\le i\le n}\T_{i,i-1}$ the matrix having $1$'s on the subdiagonal.  Likewise, we denote by $\U=\sum_{1\le i\le n-1}\T_{i,i+1}$ the matrix having $1$'s on the superdiagonal. 
\end{notation}

Notice that $\LL^m=\sum_{m+1\le i\le n}\T_{i,i-m}$ and $\LL^m=0$ for $m\ge n$. Similarly, we have $\U^m=\sum_{1\le i\le n-m}\T_{i,i+m}$ and $\U^m=0$ for $m\ge n$. By convention, $\LL^0=\U^0=\I$.

\begin{lemma}\label{lem:matrix}
For any $\M\in M_n(\kk)$, we have the following statements:
\begin{enumerate}
\item If $(\sum_{i\ge 0} a_i\LL^i)\M=\M(\sum_{i\ge 0} a_i\LL^i)$ for $a_i \in\kk$ with $a_1\neq 0$, then $\M=\sum_{i\ge 0} p_i\LL^i$ for some $p_i\in\kk$.
\item If $(\sum_{i\ge 0} b_i\U^i)\M=\M(\sum_{i\ge 0} b_i\U^i)$ for $b_i \in \kk$ with $b_1\neq 0$, then $\M=\sum_{i\ge 0} q_i\U^i$ for some $q_i\in\kk$.
\item Let $\psi$ be an endomorphism of $M_n(\kk)$ defined by $$\M\mapsto \left(\textstyle \sum_{i\ge 0} a_i\LL^i\right)\M-\M\left( \textstyle \sum_{i\ge 0} b_i\U^i\right).$$ If $a_0\neq b_0$, then $\psi$ is bijective.
\item Let $\phi$ be an endomorphism of $M_n(\kk)$ defined by 
$$\M\mapsto \left(\textstyle \sum_{i\ge 0} b_i\U^i\right)\M-\M\left(\textstyle \sum_{i\ge 0} a_i\LL^i\right).$$  If $a_0\neq b_0$, then $\phi$ is bijective.
\end{enumerate}
\end{lemma}

\begin{proof}
It suffices to show (a) and (c); the proofs of (b) and (d) are similar. 

(a) It is clear that the minimal polynomial of $\sum_{i\ge 0} a_i\LL^i$ is $(x-a_0)^{n}$ since $a_1\neq 0$. So, after a linear transformation, we can assume that $a_2=a_3=\cdots =a_{n-1}=0$ and $a_1\neq 0$. Then, the result follows from a direct computation. 

(c) To prove the injectivity of $\psi$, say $\X:=(\sum_{i\ge 0} a_i\LL^i)\M-\M(\sum_{i\ge 0} b_i\U^i)=0$ for $\M=:(m_{ij})$. We proceed by induction on the size of the matrices, $n$. If $n=1$, then the result clearly holds. Further, for $n >1$, it is clear that the $(1,1)$-th entry of $\X$ is $(a_0-b_0)m_{11}$. So, $m_{11}=0$ since $a_0 \neq b_0$. Then, the $(1,2)$-th entry of $\X$ is $(a_0-b_0)m_{12}=0$, which again implies that $m_{12}=0$. The same is true for $m_{21}$. Continuing in this manner, we see that the first column and the first row of $\M$ are zero. Now injectivity follows from induction. Since $M_n(\kk)$ is finite dimensional, the map $\psi$ is also surjective. 
\end{proof}

Let us study the  quantum linear groups ${\mathcal O}^{c}_{A}(GL/S^\infty)$ and ${\mathcal O}_{A}(SL/S^\infty)$, for $A$ an AS regular algebra of Jordan type. Compare to Example~\ref{ex:B(AJ)}.

\begin{proposition}\label{prop:GLJord}
For any AS regular algebra of Jordan type $A_{Jord}=A(\J_n)$, we have that 
\begin{enumerate}
\item ${\mathcal O}^{c}_{A_{Jord}}(GL/S^{2m})$ is generated by $a_0,a_1,\dots,a_{n-1}$ and central group-like elements $\sf D^{\pm 1}$, subject to the relations:
\[
a_0^2={\sf D} \quad \quad \text{and}
 \quad\quad
a_0a_{i-1}-a_1a_{i-2}+\cdots+(-1)^{i-1}a_{i-1}a_0=0 \quad \text{ ~~for all $2 \leq i \leq n$}
\]
\noindent with the coalgebra structure and antipode given by
\[
\begin{array}{c}
\Delta(a_0)=a_0\otimes a_0\quad\text{and}
\quad 
\Delta(a_{i-1})=a_0\otimes a_{i-1}+a_1\otimes a_{i-2}+\cdots+a_{i-1}\otimes a_0 \quad \text{ ~~for all $2 \leq i \leq n$}\\
\ee(a_{i})=\delta_{i0} \quad \text{ and  } \quad S(a_i)=(-1)^i{\sf D}^{-1}a_i \text{ ~~for all $0\le i\le n-1$;}
\end{array}
\]
\item ${\mathcal O}^{c}_{A_{Jord}}(GL/S^\infty)\cong {\mathcal O}^{c}_{A_{Jord}}(GL/S^2)$ as Hopf algebras, and ${\mathcal O}^{c}_{A_{Jord}}(GL/S^\infty)$ is involutory.
\end{enumerate}
\end{proposition}

\begin{proof}
(a) By Corollary \ref{C:SLGL}, we know that ${\mathcal O}^{c}_{A_{Jord}}(GL)\cong \mc{G}(\J_n^{-1}, \J_n)$ for $\E=\J_n$. Suppose that ${\mathcal O}^{c}_{A_{Jord}}(GL)$ is generated by entries of $\mathbb A:=(a_{i,j})_{1 \leq i, j \leq n}$ and $\sf D^{\pm 1}$. Moreover, we have that $S^2(\sf D)=\sf D$ and
$S^2(\mathbb A)=\X\mathbb A \X^{-1}$, where $\X:=\J_n(\J_n^{-1})^T$.
The Hopf ideal $\mc{L}_{2m}$ that defines the $S^{2m}$-trivial general linear quantum group $\mathcal O^{c}_{A_{Jord}} (GL/S^{2m})$ is generated by entries of the matrix:
$\X^m\mathbb A-\mathbb A\X^m$.
Using Notation~\ref{not:e,g,h}, direct computation shows that
\[
\X=(-1)^{n+1}\LL^0+2(-1)^n\LL+\cdots+2(-1)^2\LL^{n-1}.
\]
Hence, we have that $$\X^{m}=(-1)^{m(n+1)}\LL^0+2m(-1)^{nm+m-1}\LL+\cdots.$$ 
With a dual version of Lemma \ref{lem:matrix}(a), we get that $a_{i,j}=:a_{i-j}$ if $j\leq i$, and $a_{i,j} = 0$ if $j > i$ in the quotient space $\overline{\mathbb A}:=\mathbb A/(\X^m\mathbb A-\mathbb A\X^m)$.
So we take representatives $a_0,a_1,\dots,a_{n-1}$ such that
\begin{equation} \label{eq:Abar}
\overline{\mathbb A}=a_0\LL^0+a_1\LL+a_2\LL^2+\cdots+a_{n-1}\LL^{n-1}.
\end{equation}
Now using the structure of $\mc{G}(\J_n^{-1},\J_n)$ in Definition~\ref{def:B(E)}, one obtain the relations, coalgebra structure, and antipode of ${\mathcal O}^{c}_{AJord}(GL/S^{2m})$ as claimed by direct computation. 

(b) This follows from (a) since ${\mathcal O}^{c}_{A_{Jord}}(GL/S^{2m})$ stabilizes for $m\ge 1$.
\end{proof}

We get the following immediate consequence.
\begin{corollary}\label{Cor:SLJord}
For any AS regular algebra of Jordan type $A_{Jord}=A(\J_n)$, we have that 
\begin{enumerate}
\item  ${\mathcal O}_{A_{Jord}}(SL/S^{2m})$ is generated by $a_0,a_1,\dots,a_{n-1}$, subject to the relations:
\[
a_0^2=1 \quad\quad\text{and}
 \quad\quad
a_0a_{i-1}-a_1a_{i-2}+\cdots+(-1)^{i-1}a_{i-1}a_0=0 \quad \text{ ~~for all $2 \leq i \leq n$}
\]
\noindent with the coalgebra structure and antipode given by
\[
\begin{array}{c}
\Delta(a_0)=a_0\otimes a_0\quad\text{and}
\quad 
\Delta(a_{i-1})=a_0\otimes a_{i-1}+a_1\otimes a_{i-2}+\cdots+a_{i-1}\otimes a_0 \quad \text{ ~~for all $2 \leq i \leq n$}\\
\ee(a_{i})=\delta_{i0} \quad \text{ and  } \quad S(a_i)=(-1)^ia_i\text{ ~~for all $0\le i\le n-1$;}
\end{array} 
\] 
\item ${\mathcal O}_{A_{Jord}}(SL/S^\infty)\cong {\mathcal O}_{A_{Jord}}(SL/S^2)$ as Hopf algebras, and ${\mathcal O}_{A_{Jord}}(SL/S^\infty)$ is involutory.
\qed
\end{enumerate}
\end{corollary}

Now let us study the  quantum linear groups ${\mathcal O}^{c}_{A}(GL/S^\infty)$ and ${\mathcal O}_{A}(SL/S^\infty)$, for $A$ an AS regular algebra of double quantum type. Compare to Example~\ref{ex:B(Aq)}.

\begin{proposition}\label{prop:GLDq}
For any AS regular algebra of double quantum type $A_{D_q}=A(\DD_{2r}(q))$ with $q$ a non-root of unity, we have the statements below.
\begin{enumerate}
\item $\mc{O}^{c}_{A_{D_q}}(GL/S^{2m}) $  is generated by $a_0,a_1,\dots a_{r-1},$ $b_0,b_1,\dots, b_{r-1}$ and central group-like elements $\sf D^{\pm 1}$
subject to relations:
\[
\begin{array}{c}
a_0b_0~=~b_0a_0~=~\sf D\\
$$a_0b_{i-1}+a_1b_{i-2}+\cdots+a_{i-1}b_0~=~b_0a_{i-1}+b_1a_{i-2}+\cdots+b_{i-1}a_0~=~0
\quad \text{ ~~for all $2 \leq i \leq r$}
\end{array}
\]
with the coalgebra structure and antipode given by
\[
\begin{array}{lll}
\Delta(a_0)=a_0\otimes a_0, \quad &\Delta(a_{i-1})=a_0\otimes a_{i-1}+a_1\otimes a_{i-2}+\cdots+a_{i-1}\otimes a_0 \quad &\text{ for $2 \leq i \leq r$}\\ 
 \Delta(b_0)=b_0\otimes b_0, \quad &\Delta(b_{i-1})=b_0\otimes b_{i-1}+b_1\otimes b_{i-2}+\cdots+b_{i-1}\otimes b_0 \quad &\text{ for $2 \leq i \leq r$},
\end{array} 
\]
where $\ee(a_i)=\ee(b_i)=\delta_{i0}$ and $S(a_i)={\sf D^{-1}}b_i,~ S(b_i)={\sf D^{-1}}a_i$ for all $0\le i\le r-1$.
\item ${\mathcal O}^{c}_{A_{D_q}}(GL/S^\infty)\cong \mathcal{O}^{c}_{A_{D_q}}(GL/S^2)$ as Hopf algebras, and ${\mathcal O}^{c}_{A_{D_q}}(GL/S^\infty)$ is involutory.
\end{enumerate}
\end{proposition}

\begin{proof}
(a) We apply the same process as in the proof of Proposition~\ref{prop:GLJord} for $\E= \DD_{2r}(q)$. We see that $\E(\E^{-1})^T$=$\Diag(\X,\Y)$, where $\X=q^{-1}\I-q^{-2}\LL$ and $\Y=q\I+\U$ are $r$-by-$r$ matrices, using Notation~\ref{not:e,g,h}. Write $\A=(\A_{ij})_{1 \leq i,j \leq 2}$, for $\A_{ij}$ being an $r$-by-$r$ matrix. Then, the subspace spanned by the entries of $S^{2m}(\mathbb A)-\mathbb A$ can be spanned by the entries of
{\small \begin{align*}
{\begin{pmatrix} \X&0\\ 0&\Y\end{pmatrix}}^m\mathbb A-\mathbb A{\begin{pmatrix} \X&0\\ 0&\Y\end{pmatrix}}^m=&\begin{pmatrix} \X^m&0\\ 0&\Y^m\end{pmatrix}\begin{pmatrix} \A_{11}&\A_{12}\\ \A_{21}&\A_{22}\end{pmatrix}-\begin{pmatrix} \A_{11}&\A_{12}\\ \A_{21}&\A_{22}\end{pmatrix}\begin{pmatrix} \X^m&0\\ 0&\Y^m\end{pmatrix}\\
=&\begin{pmatrix} \X^m\A_{11}-\A_{11}\X^m&\X^m\A_{12}-\A_{12}\Y^m\\ \Y^m\A_{21}-\A_{21}\X^m&\Y^m\A_{22}-\A_{22}\Y^m\end{pmatrix}.
\end{align*}}

\vspace{-.15in}

\noindent Note that $$\X^m=q^{-m}\I-mq^{-m-1}\LL+\cdots \quad \text{ and } \quad \Y^m=q^m\I+mq^{m-1}\U+\cdots.$$ Since $q^{2m} \neq 1$, we have that $q^{-m}\neq q^m$. Hence $\A_{12}/(\X^m\A_{12}-\A_{12}\Y^m)=0$ by Lemma \ref{lem:matrix}(c). Similarly, $\A_{21}/(\Y^m\A_{21}-\A_{21}\X^m)=0$ by Lemma \ref{lem:matrix}(d). Furthermore by applying Lemma \ref{lem:matrix}(a,b), we can choose representatives $a_0$, $a_1$, $\dots$, $a_{r-1}$, $b_0$, $b_1$, $\dots$, $b_{r-1}$ in the quotient space $\overline{\mathbb A}=\Diag(\overline{\A_{11}},\overline{\A_{22}})$ such that 
\[
\begin{array}{c}
\overline{\A}_{11}:=\A_{11}/( \X^m\A_{11}-\A_{11}\X^m) =a_0\LL^0+a_1\LL+\cdots+a_{r-1}\LL^{r-1}\\
 \overline{\A}_{22}:=\A_{22}/( \Y^m\A_{22}-\A_{22}\Y^m) =b_0\U^0+b_1\U+\cdots+b_{r-1}\U^{r-1}.
 \end{array}
 \]
Recall  $\E =\mathbb D_{2r}(q)=${\footnotesize$\begin{pmatrix} 0&\I_r 
\\\B_r(q)&0\end{pmatrix}$} as in Notation~\ref{not:matrices}. After passing to the quotient algebra, the relations of $\mathcal{O}^{c}_{A_{D_q}}(GL/S^{2m})$ are $\E{\overline{\mathbb A}}^T\E^{-1}\overline{\mathbb A}={\sf D}\I=\overline{\mathbb A}\E{\overline{\mathbb A}}^T\E^{-1}$.
Hence, we have 
$$ \overline{\A}_{11}\overline{\A}_{22}^T=\overline{\A}_{22}^T \overline{\A}_{11}={\sf D}\I \quad \text{ and } \quad \B_r(q) \overline{\A}_{11}^T\B_r(q)^{-1}\overline{\A}_{22}=\overline{\A}_{22}\B_r(q) \overline{\A}_{11}^T\B_r(q)^{-1}={\sf D}\I.$$ 
Since  $\overline{\A}_{11}^T\B_r(q)=\B_r(q)\overline{\A}_{11}^T$, the relations above are equivalent to $\overline{\A}_{22}\overline{\A}_{11}^T=\overline{\A}_{11}^T \overline{\A}_{22}={\sf D}\I$. Now the result follows by a direct computation. Namely, if $(a_{i,j})_{1 \leq i,j \leq 2r}$ are the generators for $\mc{O}^{c}_{A_{D_q}}(GL)$, then $a_{i,j}=a_{i-j}$ for $1\leq j \leq i \leq r$,  and $a_{i,j}=b_{j-i}$ for $r+1 \leq i \leq j \leq 2r$, and $a_{i,j} = 0$ otherwise.

(b) This follows from (a). Namely, we see that ${\mathcal O}^{c}_{A_{D_q}}(GL/S^{2m})$ stabilizes for $m\ge 1$.
\end{proof}

We get the following immediate consequence.

\begin{corollary}\label{Cor:SLDq}
For any AS regular algebra of double quantum type $A_{D_q}=A(\DD_{2r}(q))$ with $q$ a non-root of unity, we have that:
\begin{enumerate}
\item $\mc{O}_{A_{D_q}}(SL/S^{2m}) $  is generated by $a_0,a_1,\dots a_{r-1},$ $b_0,b_1,\dots, b_{r-1}$
subject to relations:
\[
\begin{array}{c}
a_0b_0~=~b_0a_0~=~1\\
$$a_0b_{i-1}+a_1b_{i-2}+\cdots+a_{i-1}b_0~=~b_0a_{i-1}+b_1a_{i-2}+\cdots+b_{i-1}a_0~=~0
\quad \text{ ~~for all $2 \leq i \leq r$}
\end{array}
\]
with the coalgebra structure and antipode given by
\[
\begin{array}{lll}
\Delta(a_0)=a_0\otimes a_0, \quad &\Delta(a_{i-1})=a_0\otimes a_{i-1}+a_1\otimes a_{i-2}+\cdots+a_{i-1}\otimes a_0 \quad &\text{ for $2 \leq i \leq r$}\\ 
 \Delta(b_0)=b_0\otimes b_0, \quad &\Delta(b_{i-1})=b_0\otimes b_{i-1}+b_1\otimes b_{i-2}+\cdots+b_{i-1}\otimes b_0 \quad &\text{ for $2 \leq i \leq r$},
\end{array} 
\]
where $\ee(a_i)=\ee(b_i)=\delta_{i0}$ and $S(a_i)=b_i, S(b_i)=a_i$ for all $0\le i\le r-1$;
\item ${\mathcal O}_{A_{D_q}}(SL/S^\infty)\cong \mathcal{O}_{A_{D_q}}(SL/S^2)$ as Hopf algebras, and ${\mathcal O}_{A_{D_q}}(SL/S^\infty)$ is involutory.
\qed
\end{enumerate}
\end{corollary}

Now we provide an alternative description for the  quantum linear groups ${\mathcal O}^{c}_{A}(GL/S^\infty)$ and ${\mathcal O}_{A}(SL/S^\infty)$, for $A_{Jord}$, and for $A_{D_q}$ with $q$ a non-root of unity. First, consider the following terminology.

\begin{definition}[{\sf NSymm}($n$)] \cite[Section~3.1]{GKLLRT}
The {\it Hopf algebra of noncommutative symmetric functions on $n$ variables} {\sf NSymm}($n$) is the free algebra on $n$ generators $\kk\langle x_1,\dots,x_n\rangle$, with coalgebra structure  given by
$\Delta(x_i)=\sum_{j+k=i}x_j\otimes x_k$. Here, $x_0 = 1$.
Moreover, $\ee(x_i)=\delta_{i0}$, and the antipode $S$ is defined inductively: $S(x_1)=-x_1$ with $S(x_i)=-\sum_{k=1}^i x_kS(x_{i-k})$  for $2\le i\le n$. 
\end{definition}

We have that 
\begin{equation} \label{eq:Nsymm}
{\sf NSymm}(n) \cong U(\mc{L}\langle p_1, \dots, p_n \rangle),
\end{equation}
 where the latter is the universal enveloping algebra of the free Lie algebra generated by $p_1, \dots, p_n$, where the $p_i$s are defined inductively by $p_1=x_1$ and $ix_i = x_{i-1} p_1 + x_{i-2} p_2 + \dots + p_i$ for $2\le i\le n$.
 
\begin{corollary} \label{cor:JordInfty}
We have the following isomorphisms of Hopf algebras:
\[\mc{O}^{c}_{A_{Jord}}(GL/S^\infty)\cong {\sf NSymm}(n\text{-}1)\#\kk\mathbb{Z}
\quad \quad \text{and} \quad \quad \mc{O}_{A_{Jord}}(SL/S^\infty)\cong {\sf NSymm}(n\text{-}1) \#\kk[\mathbb Z/2\mathbb Z].\]
Here, the action of $\kk\mathbb{Z}$ (resp., of  $\kk[\mathbb Z/ 2\mathbb Z]$) on {\sf NSymm}$(n$-$1)$ is given inductively by $g(x_1)=x_1$ with $g(x_i)=\sum_{j=1}^i (-1)^{j+1}x_jg(x_{i-j})$, for $2\le i\le n-1$ for  the generator $g$ of $\mathbb Z$ (resp., of $\mathbb Z/2\mathbb Z$).
\end{corollary}
\begin{proof}
This follows from Proposition \ref{prop:GLJord} and Corollary \ref{Cor:SLJord} by using the isomorphism: $g\mapsto a_0$ and $x_i\mapsto a_ia_0^{-1}$ for $1\le i\le n-1$.
\end{proof}

Take $\mathbb Z$ copies of the noncommutative symmetric functions ${\sf NSymm}(n)^{[s]}$ on $n$ variables $x_1^{[s]},\dots, x_n^{[s]}$, indexed by $s\in \mathbb Z$. We consider the free product $${\sf FNSymm}(n):=\coprod_{s\in \mathbb{Z}}{\sf NSymm}(n)^{[s]}.$$ 
We use $\varphi$ to denote the Hopf automorphism of each summand ${\sf NSymm}(n)^{[s]}$ given inductively by 
$$\varphi(x_1^{[s]})=-x_1^{[s]} \quad \text{and} \quad \varphi(x_i^{[s]})=-\sum_{1\le j\le i} x_j^{[s]}\varphi(x_{i-j}^{[s]}), \quad \text{for } 2\le i \le n,$$ where $x_0^{[s]}=1$. It is easy to show that $\varphi$ is well-defined and $\varphi^2=\text{id}$, by induction. To see this, note that
$$\sum_{\substack{i+j=m\\ i,j>0}} x_i \varphi(x_j) = \sum_{\substack{i+j=m\\ i,j>0}} \varphi(x_i) x_j$$
\noindent for all $m \geq 2$.
We denote by $(g,h)$ the generators of $\mathbb Z\times \mathbb Z$. The $\mathbb Z\times \mathbb Z$-action on ${\sf FNSymm}(n)$ is given by 
\[
g^k(x_i^{[s]})=\hat{\varphi}^k(x_i^{[s+k]}) \quad \text{and} \quad  h^k(x_i^{[s]})=\check{\varphi}^k(x_i^{[s-k]}),\quad \text{for } s\in \mathbb Z,
\]
where $\hat{\varphi} = \check{\varphi} = \varphi$.

\begin{corollary} \label{cor:DqInfty}
Suppose that $q$ is not a root of unity. Then, we have the Hopf algebra isomorphisms below:
\[\mc{O}^{c}_{A_{D_q}}(GL/S^\infty)\cong {\sf FNSymm}(r\text{-}1)\#\kk[\mathbb Z\times \mathbb Z]\quad \text{and} \quad \mc{O}_{A_{D_q}}(SL/S^\infty)\cong {\sf FNSymm}(r\text{-}1)\#\kk\mathbb{Z},\]
where the $\mathbb Z$-action on ${\sf FNSymm}(r\text{-}1)$ in the latter isomorphism is induced by taking $gh=1$ in the $\mathbb Z\times \mathbb Z$-action on ${\sf FNSymm}(r\text{-}1)$.
\end{corollary}
\begin{proof}
We denote the generators of $\mc{O}^{c}_{A_{D_q}}(GL/S^\infty)$ by $a_0,\dots,a_{r-1},b_0,\dots,b_{r-1}$ and ${\sf D}^{\pm 1}$ subject to the relations in Proposition \ref{prop:GLDq}. Recall, in particular, that $a_0b_0=b_0a_0 = {\sf D}$. We define a map 
\begin{equation*} \label{eq:Psi}
\Psi: \mc{O}^{c}_{A_{D_q}}(GL/S^\infty)\to {\sf FNSymm}(r\text{-}1)\#\kk[\mathbb Z\times \mathbb Z]
\end{equation*} given by
\[
\Psi(a_0)=g,\quad \Psi(b_0)=h, \quad \Psi(a_i)=x_i^{[1]}g, \quad\Psi(b_i)=x_i^{[0]}h, \quad \mbox{for all}\ 1\le i\le r-1.
\]
It follows from a direct computation that $\Psi$ is a well-defined Hopf algebra map. Now $\Psi$ maps the Hopf subalgebra $A:=\kk\langle a_1a_0^{-1},\dots,a_{r-1}a_0^{-1}\rangle$ of $\mc{O}^{c}_{A_{D_q}}(GL/S^\infty)$ onto the $[1]$-th copy ${\sf NSymm}(r\text{-}1)^{[1]}$, and $B:=\kk \langle b_1b_0^{-1},\dots,b_{r-1}b_0^{-1}\rangle$ onto the $[0]$-th copy ${\sf NSymm}(r\text{-}1)^{[0]}$, in ${\sf FNSymm}(r\text{-}1)$. 

By an abuse of notation, we take $\hat{\varphi}(a_1a_0^{-1})=-a_1a_0^{-1}$, and $\hat{\varphi}(a_ia_0^{-1})=-\sum_{1\le j\le i} a_ja_0^{-1}\hat{\varphi}(a_{i-j}a_0^{-1})$ for $2\le i \le r-1$ inductively; we take the same for $B$ using $\check{\varphi}$. Note that $\varphi^2=\text{id}$ on $A,B$. Then, direct computation shows that $\Psi$ is bijective with inverse given by $\Psi^{-1}(g)=a_0,~~\Psi^{-1}(h)=b_0$, and for $1\le i\le r-1$, 
\begin{align*}
\Psi^{-1}(x_i^{[s]})=
\begin{cases}
b_0^{-s}\check{\varphi}^{s}(b_ib_0^{-1})b_0^s  & s\leq 0\\
a_0^{s-1}\hat{\varphi}^{1-s}(a_ia_0^{-1})a_0^{1-s}  & s>0.
\end{cases}
\end{align*}
\noindent Here, $a_0^{-1}=b_0{\sf D}^{-1}$ and $b_0^{-1}=a_0{\sf D}^{-1}$, and for $t \geq 0$ note that 
\begin{equation} \label{eq:Psi1}
\Psi^{-1}(\varphi^t(x_i^{[s]}))=
\begin{cases}
b_0^{-s}\check{\varphi}^{s+t}(b_ib_0^{-1})b_0^s  & s\leq 0\\
a_0^{s-1}\hat{\varphi}^{1-s+t}(a_ia_0^{-1})a_0^{1-s}  & s>0.
\end{cases}
\end{equation} 
The most difficult calculations are the following; the rest are routine or follow similarly to those below:
\[
\begin{array}{lllll}
\Psi^{-1}(g x_i^{[s]}) &= 
\begin{cases}
a_0 b_0^{-s} \check{\varphi}^s(b_ib_0^{-1})b_0^s & s \leq 0\\
a_0 a_0^{s-1} \hat{\varphi}^{1-s}(a_ia_0^{-1})a_0^{1-s} & s > 0\\
\end{cases} \\\\
&= 
\begin{cases}
a_0 b_0^{-s} \check{\varphi}^s(b_ib_0^{-1})b_0^s & s <-1\\
a_0 b_0 \check{\varphi}^{-1}(b_ib_0^{-1})b_0^{-1} & s = -1\\
a_0 b_ib_0^{-1} & s = 0\\
a_0 a_0^{s-1} \hat{\varphi}^{1-s}(a_ia_0^{-1})a_0^{1-s} & s > 0\\
\end{cases} 
&=
\begin{cases}
{\sf D} b_0^{-(s+1)} \check{\varphi}^s(b_ib_0^{-1})b_0^{s+1}a_0 {\sf D}^{-1} & s \leq -1\\
{\sf D} \check{\varphi}(b_ib_0^{-1})b_0^{-1} & s = -1\\
a_0 b_ib_0^{-1} (a_0^{-1}a_0) & s =0\\
a_0^{s} \hat{\varphi}^{1-s}(a_ia_0^{-1})a_0^{-s} a_0 & s > -1\\
\end{cases} \\\\
&=
\begin{cases}
b_0^{-(s+1)} \check{\varphi}^s(b_ib_0^{-1})b_0^{s+1}a_0  & s \leq -1\\
\check{\varphi}(b_ib_0^{-1}) {\sf D}  b_0^{-1} & s = -1\\
(a_0 (b_ib_0^{-1}) a_0^{-1}) a_0  & s =0\\
a_0^{s} \hat{\varphi}^{1-s}(a_ia_0^{-1})a_0^{-s} a_0  & s > -1\\
\end{cases} 
&=
\begin{cases}
b_0^{-(s+1)} \check{\varphi}^{s+2}(b_ib_0^{-1})  b_0^{s+1} a_0& s \leq -1\\
\check{\varphi}(b_ib_0^{-1})  a_0& s = -1\\
\hat{\varphi}(a_ia_0^{-1}) a_0 & s =0\\
a_0^{s} \hat{\varphi}^{1-s}(a_ia_0^{-1})a_0^{-s} a_0 & s > -1\\
\end{cases} \\\\
&=
\begin{cases}
b_0^{-(s+1)} \check{\varphi}^{s+2}(b_ib_0^{-1})  b_0^{s+1} a_0& s \leq -1\\
a_0^{s} \hat{\varphi}^{1-s}(a_ia_0^{-1})a_0^{-s} a_0 & s > -1\\
\end{cases} 
&= \Psi^{-1}\left(\hat{\varphi}(x_i^{[s+1]})g\right);
\end{array}
\]
\begin{align*}
\Psi(\Psi^{-1}(x_i^{[s]})) \hspace{.15in}&=
\begin{cases}
\Psi\left(b_0^{-s}\check{\varphi}^{s}(b_ib_0^{-1})b_0^s\right)  & s\leq 0\\
\Psi\left(a_0^{s-1}\hat{\varphi}^{1-s}(a_ia_0^{-1})a_0^{1-s}\right)  & s>0
\end{cases}
&=
\begin{cases}
h^{-s}\check{\varphi}^{s}(x_i^{[0]})h^s & s\leq 0\\
g^{s-1}\hat{\varphi}^{1-s}(x_i^{[1]})g^{1-s}  & s>0
\end{cases}
\hspace{.15in}&= x_i^{[s]}.
\end{align*}
 This establishes the first isomorphism; the argument for the second isomorphism is similar by using Corollary~\ref{Cor:SLDq}.
\end{proof}

In the following, we consider the Noetherian condition and Gelfand-Kirillov (GK) dimension for the quantum group $\mc{O}_{A(\E)}(SL/S^2)$, which is a Hopf quotient of both $\mc{O}_{A(\E)}(GL/S^2)$ and $\mc{O}^c_{A(\E)}(GL)$.

\begin{lemma}\label{L:NonNoeth}
Take $\E = \DD_{2r}(q)$  for $r \geq 2$, or $\J_n$ for $n \geq 3$. Then, we have that $\mc{O}_{A(\E)}(SL/S^2)$ is not Noetherian and has infinite GK dimension. 
\end{lemma}

\begin{proof}
First, consider the case $\E=\DD_{2r}(q)$. As in the proof of Proposition \ref{prop:GLDq}, take the quotient space 
$
\overline{\mathbb A}={\footnotesize\begin{pmatrix} \overline{\A}_{11}&0 \\0& \overline{\A}_{22} \end{pmatrix}} 
$
of the generating space $\A$ of $\mc{O}_{A_{D_q}}(SL)\cong \mc{B}(\E^{-1})$,
where 
\begin{align*}
\overline{\A}_{11}:=a_0\LL^0+a_1\LL+\cdots+a_{r-1}\LL^{r-1}\quad \text{ and } \quad \overline{\A}_{22}:=b_0\U^0+b_1\U+\cdots+b_{r-1}\U^{r-1}.
\end{align*}
The corresponding quotient Hopf algebra of $\mc{B}(\E^{-1})$ is generated by 
$a_0, \dots, a_{r-1}, b_0, \dots, b_{r-1}$, 
subject to the relations described in Corollary \ref{Cor:SLDq}. It is clear that this quotient Hopf algebra of $\mc{O}_{A_{D_q}}(SL)$ indeed factors through $\mc{O}_{A_{D_q}}(SL/S^2)$ since it is cocommutative and hence involutory. Then, we get a Hopf algebra surjection from $\mc{O}_{A_{D_q}}(SL/S^2)$ to ${\sf FNSymm}(r\text{-}1)\#\kk\mathbb{Z}$ via an argument similar to the proof of Corollary~\ref{cor:DqInfty}. Since $r\geq 2$, we see that ${\sf FNSymm}(r\text{-}1)$ is a free algebra on an infinite number of generators, hence it is not Noetherian and has infinite GK dimension. Now the skew group algebra  ${\sf FNSymm}(r\text{-}1)\#\kk\mathbb{Z}$ is not Noetherian since its coefficient ring is not Noetherian.
Moreover, ${\sf FNSymm}(r\text{-}1)\#\kk\mathbb{Z}$ has infinite GK dimension since it contains the subalgebra ${\sf FNSymm}(r\text{-}1)$. 

The case where $\E = \J_n$ follows in the same fashion by using  Corollary \ref{cor:JordInfty}.
\end{proof}

Now we consider the Noetherian property and GK dimension of a free product of quantum special linear groups. Let $A=\kk\langle x_1,\dots,x_n\rangle/(\mathcal I)$ and $B=\kk\langle y_1,\dots,y_m\rangle/(\mathcal J)$ be two finitely presented algebras. Recall that their free product $A\ast B$ ($=A\coprod B$) can be presented as $\kk\langle x_1,\dots,x_n,y_1,\dots,y_m\rangle/(\mathcal I,\mathcal J)$. 

\begin{lemma}\label{lem:Freeproduct} 
Let  $\mathcal{O}_{A_i}(SL/S^2)$ be either the quantum $S^2$-trivial special linear group associated  to $A_i = A(\E_i)$ for $\E_i \in GL_{n_i}(\kk)$, or $\kk[\mathbb{Z}/2\mathbb{Z}]$ associated to $A_i=A(\J_1)$. If $n_1 + \cdots + n_\ell\geq 3$,  then the free product of quantum groups $\coprod_{i=1}^\ell \mathcal{O}_{A_i}(SL/S^2)$ is not Noetherian and has infinite Gelfand-Kirillov dimension.
\end{lemma}

\begin{proof}
Recall from Remark~\ref{rem:B(J1)} that $\mathcal{O}_{A(\J_1)}(SL) \cong \kk[\mathbb{Z}/2\mathbb{Z}]$, as Hopf algebras; it is easy to see that in this case $\mathcal{O}_{A(\J_1)}(SL) = \mathcal{O}_{A(\J_1)}(SL/S^2)$.
Note for arbitrary Hopf algebras $H$ and $H'$, we get two surjective algebra maps from the free product $H\ast H'$ to $H$ and $H'$, respectively, by using the counits of $H'$ and $H$. Therefore, if $H\ast H'$ is Noetherian and has finite GK dimension, then so does  $H$ and $H'$. Thus, by Lemma \ref{L:NonNoeth} and Corollary~\ref{C:SLGL}, it suffices to show that the $S^2$-trivial quotient Hopf algebras of the following free products of quantum groups are not Noetherian and have infinite GK dimension: 
\begin{center}
\begin{tabular}{rlllllrl}
(I)& $\mc{B}(\DD_2(q)^{-1}) \ast \mc{B}(\DD_2(p)^{-1})$; &&&&&
(IV)& $\mc{B}(\J_2^{-1}) \ast \mc{B}(\J_1^{-1})$;\\
(II)& $\mc{B}(\DD_2(q)^{-1}) \ast \mc{B}(\J_1^{-1})$; &&&&&
(V)& $\mc{B}(\J_2^{-1}) \ast \mc{B}(\J_2^{-1})$; \\
(III)& $\mc{B}(\DD_2(q)^{-1}) \ast \mc{B}(\J_2^{-1})$; &&&&&
(VI)& $\mc{B}(\J_1^{-1}) \ast \mc{B}(\J_1^{-1}) \ast \mc{B}(\J_1^{-1})$. \\
\end{tabular}
\end{center}
\noindent By Example~\ref{ex:B(Aq)}, we have that $\mc{B}(\DD_2(q)^{-1})$ has an $S^2$-trivial quotient Hopf algebra $\kk\mathbb{Z}$  (take $b = c= 0$). By Example~\ref{ex:B(AJ)}, we have that $\mc{B}(\J_2^{-1})$ has an $S^2$-trivial quotient Hopf algebra $\kk[\mathbb Z/2\mathbb Z] \otimes \kk[x]$ (take $a= d$, $b= 0$, $x:=c$), which in turn, has an $S^2$-trivial quotient Hopf algebra $\kk[x]$ (take, further, $a=d=1$). Hence, by taking $S^2$-trivial quotient  Hopf algebras of (I)-(VI) above, it suffices to check that the following Hopf algebras are not Noetherian and have infinite GK dimension:
\begin{center}
\begin{tabular}{rll}
(i)& $\kk\mathbb{Z}\ast \kk\mathbb{Z}$ & $\cong \kk\langle x, x^{-1}, y, y^{-1}\rangle /(x x^{-1} -1, ~x^{-1} x -1, ~y y^{-1} -1, ~ y^{-1} y -1)$;\\
(ii)& $\kk\mathbb{Z}\ast\kk[\mathbb Z/2\mathbb Z]$ & $\cong \kk\langle x, x^{-1}, y\rangle /(x x^{-1} -1, ~x^{-1} x -1, ~y^2 -1)$;\\
(iii)& $\kk[x] \ast \kk\mathbb{Z}$ & $\cong \kk\langle x, y, y^{-1}\rangle /(y y^{-1} -1, ~ y^{-1} y -1)$;\\
(iv)& $\kk[x]\ast \kk[\mathbb Z/2\mathbb Z]$ & $\cong \kk\langle x, y \rangle /(y^2 -1)$;\\
(v)& $\kk[x]\ast \kk[y]$ & $\cong \kk\langle x, y \rangle$;\\
(vi)& $\kk[\mathbb Z/2\mathbb Z]\ast \kk[\mathbb Z/2\mathbb Z] \ast\kk[\mathbb Z/2\mathbb Z]$ & $\cong \kk\langle x, y, z\rangle /(x^2-1,~y^2 -1,~z^2-1)$.\\
\end{tabular}
\end{center}
To this end, consider the ascending chain of ideals $I_1 \subsetneq I_2 \subsetneq \dots \subsetneq I_j \subsetneq \dots$, in the respective cases:
\[
\begin{array}{llllll}
\text{for (i)}& I_j =(yxy+x, & yx^2y+x^2,& \dots,& yx^jy+x^j);\\
\text{for (ii)}& I_j =(yxy+x, & yx^2y+x^2,& \dots,& yx^jy+x^j);\\
\text{for (iii)}& I_j =(xyx, & xy^2x,& \dots,& xy^jx);\\
\text{for (iv)}& I_j =(yxy+x, & yx^2y+x^2,& \dots,& yx^jy+x^j);\\
\text{for (v)}& I_j =(xyx, & xy^2x,& \dots,& xy^jx);\\
\text{for (vi)}& I_j =(xyzx+yz, & x(yz)^2x+(yz)^2, & \dots, & x(yz)^jx+(yz)^j).\\
\end{array}
\]
Thus, the Hopf algebras in (I)-(VI) above are not Noetherian, as desired.

Moreover, the Hopf algebras (i)-(vi) all contain a free algebra on two variables: $\kk\langle x,y\rangle$ for (i), (iii), (v); $\kk\langle x,xy\rangle$ for (ii), (iv); and $\kk\langle xy,xz\rangle $ for (vi). Thus, the Hopf algebras in (I)-(VI) above have infinite GK dimension, as desired.
\end{proof}

\begin{theorem} \label{thm:ring}
Recall the Hopf algebra isomorphisms provided in Corollary~\ref{C:SLGL}. We have that the following statements are equivalent for any AS regular algebra $A(n,\E)$ of global dimension 2:
\begin{enumerate}
\item $n=2$.
\item $\mc{O}^{c}_{A(n,\E)}(GL) \cong \mc{G}(\E^{-1},\E)$, and $\mc{O}_{A(n,\E)}(SL) \cong\mc{B}(\E^{-1})$ are Noetherian;
\item $\mc{O}^{c}_{A(n,\E)}(GL) \cong \mc{G}(\E^{-1},\E)$, and $\mc{O}_{A(n,\E)}(SL)\cong\mc{B}(\E^{-1})$ have finite Gelfand-Kirillov dimension;
\item $\mc{O}_{A(n,\E)}(GL/S^2) \cong \mc{G}(\E^{-1},\E^T)$, and $\mc{O}_{A(n,\E)}(SL/S^2)$ are Noetherian;
\item $\mc{O}_{A(n,\E)}(GL/S^2) \cong  \mc{G}(\E^{-1},\E^T)$, and $\mc{O}_{A(n,\E)}(SL/S^2)$ have finite Gelfand-Kirillov dimension.
\end{enumerate}
\end{theorem}

\begin{proof}
(a) $\Rightarrow$ (b,c). By Corollary \ref{cor:AS}, we can assume that $\E=\DD_2(q)$ or $\J_2$. When $\E = \DD_2(q)$, we have that $\mc{O}^{c}_{A(\E)}(GL)$ is Noetherian and has finite GK dimension by \cite[Theorems~I.1.17 and~I.1.19]{book:BrownGoodearl} and \cite[Proposition~8.2.13]{MR}. (See Example~\ref{ex:B(Aq)}.) When $\E = \J_2$, we see by Example~\ref{ex:B(AJ)} that $\mc{O}^{c}_{A(\J_2)}(GL) \cong R[{\sf D}^{-1}]$, where $R=\kk[b][a;\sigma_1,\delta_1][d;\sigma_2,\delta_2][c;\sigma_3,\delta_3]$ is an iterated Ore extension of $\kk$.  Here,
\[
\begin{array}{c}
\sigma_1(b) = b, \quad \delta_1(b) = b^2, \quad \sigma_2(b) = b, \quad \sigma_2(a) = a-b, \quad \delta_2(b) = b^2, \quad \delta_2(a) = ab+b^2\\
\sigma_3(b) = b, \hspace{.12in} \sigma_3(a) = a+b, \hspace{.12in} \sigma_3(d) = b+d, \hspace{.12in}
\delta_3(b) = ba+db, \hspace{.12in} \delta_3(a)=a^2+ba-da, \hspace{.12in} \delta_3(d) = ba-da+d^2.
\end{array}
\]
So, $\mc{O}^{c}_{A(\J_2)}(GL)$ is Noetherian and has finite GK dimension. Then, $\mc{O}_{A(2,\E)}(SL)$, which is a factor of $\mc{O}^{c}_{A(2,\E))}(GL)$ by the central element ${\sf D}-1$, is Noetherian and has finite GK dimension.

(a) $\Rightarrow$ (d,e). Likewise, it is routine to see by parts (f) of Examples~\ref{ex:B(Aq)} and~\ref{ex:B(AJ)} that both  $\mc{O}_{A(\DD_2(q))}(GL/S^2)$ and $\mc{O}_{A(\J_2)}(GL/S^2)$ are isomorphic to $R[{\sf D}^{-1}]$, where $R=\kk[b][a;\sigma_1,\delta_1][d;\sigma_2,\delta_2][c;\sigma_3,\delta_3]$ is an iterated Ore extension of $\kk$. Here, ${\sf D}$ is a regular normal element of $R$. Thus, $\mc{O}_{A(\DD_2(q))}(GL/S^2)$ and $\mc{O}_{A(\J_2)}(GL/S^2)$ are Noetherian and have finite GK dimension.

(b,c) $\Rightarrow$ (a). Now take $n\ge 3$. By Lemma \ref{lem:isom} and Lemma \ref{lem:congruent}, we can assume that $\E=\E_1 \oplus \E_2 \oplus \cdots \oplus \E_\ell$, where  $\E_i$ = $\J_{n_i}$ or $\DD_{2r_i}(q_i)$, for $i = 1, \dots, \ell$. So, we have that $\mc{O}_{A(n,\E)}(SL)\cong \mc{B}(\E_1^{-1} \oplus \E_2^{-1} \oplus \cdots \oplus \E_\ell^{-1})$ by Corollary~\ref{C:SLGL}. In the generating space $\A$ of $\mc{B}(\E^{-1})$, we take the quotient space $\overline{\A}=\Diag(\overline{\A_{11}},\overline{\A_{22}},\cdots,\overline{\A_{\ell\ell}})$, so that $\overline{\A_{ii}}$ generates $\mc{B}(\E_i^{-1})$. Hence, there is a surjection from $\mc{B}(\E^{-1})$ onto the free product of Hopf algebras generated by $\overline{\A}$:
$$\overline{\mc{B}}:=\mc{B}(\E_1^{-1}) \ast \mc{B}(\E_2^{-1}) \ast \cdots \ast \mc{B}(\E_\ell^{-1}).$$ 
Now, Lemma~\ref{lem:Freeproduct} implies that the homomorphic image $\coprod_{i=1}^\ell \mathcal{O}_{A_i}(SL/S^2)$ of $\overline{\mc{B}}$  is not Noetherian and has infinite GK dimension.  Thus, $\mc{O}_{A(n,\E)}(SL)$ is not Noetherian and has infinite GK dimension. Moreover, the same holds for $\mc{O}^{c}_{A(\E)}(GL)$ by considering  its Hopf quotient $\mc{O}_{A(n,\E)}(SL)$.

(d,e) $\Rightarrow$ (a). By employing an argument to the one presented in the paragraph above, we get that $\mc{O}_{A(n,\E)}(GL/S^2)$ and $\mc{O}_{A(n,\E)}(SL/S^2)$ are not Noetherian and have infinite GK dimension since these quantum groups both have $\coprod_{i=1}^\ell \mathcal{O}_{A_i}(SL/S^2)$ as a homomorphic image.
\end{proof}


\section{Cocommutative inner-faithful Hopf coactions on AS regular algebras} \label{sec:finite}

In this section, we study inner-faithful Hopf coactions on AS regular algebras $R$. We establish a general result in Theorem~\ref{thm:centralNak}, which will be applied to the case where $R = A(\E)$ in Corollary~\ref{cor:cocom}. Consider the standing hypotheses below for this section.

\begin{hypothesis} \label{hypsec5} We assume that:
\begin{itemize}
\item $H$ is a Hopf algebra with antipode of finite order.
\item  $R$ is an AS regular algebra satisfying Hypothesis~\ref{hyp}.
\item $H$ coacts on $R$ from the right inner-faithfully.
\end{itemize}
\end{hypothesis}

Motivated by \eqref{eq:Nsymm} and Corollaries~\ref{cor:JordInfty},~\ref{cor:DqInfty}, we inquire:

\begin{question} \label{ques:Ococom}
Given the setting above, when is $H$ cocommutative? In other words, by Cartier-Kostant-Milnor-Moore theorem for cocommutative Hopf algebras, when is $H$ isomorphic to a smash product of the universal enveloping algebra of a Lie algebra and a group algebra?
\end{question}

We will see that the answer to the question above is closely related to certain property of the Nakayama automorphism of $R$ referred to as {\it \textnormal{(}power\textnormal{)} central Nakayama}; see Definition~\ref{def:centralNak} and Theorem~\ref{thm:centralNak}. We begin by considering the following subalgebras of  $M_n(\kk)$.

\begin{definition}[$\mc{C}(\X)$, $\mc{C}_P(\X)$] \label{def:cocom}
Take a matrix $\X \in M_n(\kk)$. 
\begin{enumerate}
\item The {\it centralizer of $\X$} is $\mathcal C(\X):=\{\mathbb P \in M_n(\kk)~|~\mathbb P\X=\X\mathbb P\}$. 
\item The {\it power centralizer of $\X$} is $\mathcal C_P(\X):=\bigcup_{i\ge 1} \mathcal C(\X^i)$.
\end{enumerate}
Moreover, we say that $\X, \Y \in M_n(\kk)$ are {\it $\mc{C}$-equivalent } if $\mc{C}(\X) = \mc{C}(\Y)$.
\end{definition}

Consider the following terminology.

\begin{definition} \label{def:centralNak}
Let $R$ be an AS regular algebra with Nakayama automorphism $\mu_R$. We say that:
\begin{enumerate}
\item $R$ is {\it central Nakayama} if 
$\mc{C}(\mu_R|_{R_1})$ is commutative; and
\item $R$ is {\it power central Nakayama} if 
$\mc{C}_P(\mu_R|_{R_1})$ is commutative.
\end{enumerate}
\end{definition}

Now we answer Question~\ref{ques:Ococom} in Theorem~\ref{thm:centralNak} below, but first we present a preliminary lemma.

\begin{lemma}\label{rem:decomcoalgebra}
Consider the matrix coalgebra $M_n(\kk)$, and  for $\M\in GL_n(\kk)$, the coideal 
$$\mathcal I_{\M}:=\{\M \X-\X\M~|~\X\in M_n(\kk)\}=\{\M \X\M^{-1}-\X~|~\X\in M_n(\kk)\}.$$ Then, the dual algebra of $M_n(\kk)/\mathcal I_{\M}$ is isomorphic to the centralizer algebra $\mathcal C(\M)$ of $\M$ in the matrix algebra $M_n(\kk)$.
\qed
\end{lemma}

\begin{theorem} \label{thm:centralNak}
Recall the hypotheses set at the beginning of this section and in the Introduction. Let $H$ be a Hopf algebra coacting on an AS regular algebra $R$ with homological codeterminant  ${\sf D}\in H$. If one of the following conditions hold:
\begin{enumerate}
\item[(i)] $R$ is central Nakayama, $H$ is involutory, and ${\sf D}$ is central; or
\item[(ii)] $R$ is power central Nakayama, and ${\sf D}^m$ is central for some $m \geq 1$,
\end{enumerate}
then $H$ is cocommutative.  If, further, $H$ is finite dimensional, then $H$ must be a group algebra.
\end{theorem}

\begin{proof}
We provide the proof for condition (ii) here; the proof for (i) follows in a similar fashion. Consider the Yoneda algebra $E:=\Ext_R^*(\kk,\kk)$, which is a Frobenius algebra by \cite[Corollary D]{LPWZ:Koszul}. Take a $\kk$-linear basis, say $\{x_1,\dots,x_n\}$ of $R_1$, which yields a dual basis $\{x_1^*,\dots,x_n^*\}$ of degree one part of $E$. Since $H$ coacts on $E$ from the left by \cite[Remark 1.6(d)]{CWZ:Nakayama}, we have a set of elements $\Y=\{y_{ij}\}_{1 \leq i,j \leq n}$ in $H$ such that $\rho(x_i^*)=\sum_{1\le j\le n} y_{ij}\otimes x_j^*$. Replacing \cite[Lemma 4.1]{CWZ:Nakayama} by Hypothesis~\ref{hyp}, the proof of \cite[Theorem 0.1]{CWZ:Nakayama} can be copied here as 
\[
{\sf D}^{-1}S^2(\Y){\sf D}=\M\Y\M^{-1},
\]
where $\M=(\mu_R|_{R_1})^T$. The hypothesis on ${\sf D}$ implies that there exists $m \gg 0$ such that $S^{2m}(\Y)=\M^{m}\Y\M^{-m}$. Moreover, since $S$ has finite order, we can further assume that $S^{2m}(\Y)=\M^{m}\Y\M^{-m}=\Y$.

Consider the matrix coalgebra $M_n(\kk)$ and the coideal $$\mc{I}_{m}:=\{\M^{m}\X\M^{-m}-\X~|~\X\in M_n(\kk)\}.$$ By Lemma \ref{rem:decomcoalgebra}, we can identify the dual algebra of $M_n(\kk)/\mc{I}_m$ with $\mc{C}(\M^m)$, the centralizer of the matrix $\M^{m}$ in $M_n(\kk)$. Note that $\mathcal C_P(\M^T)$ is commutative if and only if  $\mathcal C_P(\M)$ is commutative. So, since $R$ is power central Nakayama, we know that $\mc{C}(\M^m)\subseteq \mc{C}_P(\M)$ is commutative. By duality, $M_n(\kk)/\mc{I}_m$ is cocommutative. Now $\Y$ generates a subcoalgebra  $\langle \Y \rangle$ in $H$ such that $\Delta(y_{ij})=\sum_{1\le k\le n}y_{ik}\otimes y_{kj}$. Since $\M^{m}\Y\M^{-m}=\Y$, one sees that $\Y$ is a quotient coalgebra of $M_n(\kk)/\mc{I}_m$. Hence, $\langle \Y \rangle$ is cocommutative.
Thus, $H$ is also cocommutative since the $H$-coaction on $E$ is inner-faithful, and $\langle \Y \rangle$ generates a Hopf subalgebra of $H$. The last statement of the theorem is standard; see e.g. \cite[Exercise~15.3.1]{Radford}.
\end{proof}

\subsection{Conditions to yield the (power) central Nakayama property}

To apply Theorem~\ref{thm:centralNak}, it is helpful to have working conditions that imply the (power) central Nakayama property. First, we recall some well-known facts from linear algebra. 

\begin{lemma} \cite[Proposition~3.2.4 and Remark~3.2.5]{LinearAlgebra} \label{lem:centralNak}
For any matrix $\X\in M_n(\kk)$, we have that the following statements are equivalent.
\begin{enumerate}
\item $\mathcal C(\X)$ is a commutative algebra.
\item $\mathcal C(\X) = \kk[\X]$ \textnormal{(}the set of polynomials in $\X$ over $\kk$\textnormal{)}.
\item $\X$ is non-derogatory: $\text{geom.mult.}(\lambda) = \dim \ker(\X-\lambda \I_n)=1$, for all eigenvalues $\lambda$ of $\X$. 
\item $\dim \mathcal C(\X) = n$.
\item $\dim \kk[\X] = n$. 
\item In the Jordan form of $\X$, there is only Jordan block for each eigenvalue of $\X$.
\item The minimal polynomial of $\X$ and the characteristic polynomial of $\X$ are equal.
\item The minimal polynomial of $\X$ has degree $n$. \qed
\end{enumerate}
 \end{lemma}
 
Now we provide sufficient conditions for  $\mathcal C_P(\X)$ to be commutative, which should be employed with the lemma above. 
 
 \begin{definition}
Take $\X \in M_n(\kk)$ with distinct eigenvalues $\{\lambda_1, \dots, \lambda_m\}$. We say $\X$ has {\it distinct $e$-th powers of eigenvalues} if the set of scalars $\{\lambda_1^e, \dots, \lambda_m^e\}$ has distinct values, for all $e \geq 1$.
 \end{definition}

\begin{lemma} \label{lem:whenCP(X)com}
Take $\X\in M_n(\kk)$. Then,  $\mc{C}_P(\X)$ is commutative if and only if  $\X$ has distinct $e$-th powers of eigenvalues and $\mc{C}(\X)$ is commutative. In this case, we have that $\mc{C}_P(\X)=\mc{C}(\X)$.
 \end{lemma}

\begin{proof}
Suppose that $\mc{C}(\X)$ is commutative and $\X$ has distinct $e$-th powers of eigenvalues. Then, we know that $\X$ is non-derogatory by Lemma~\ref{lem:centralNak}((a) $\Rightarrow$ (c)). Since $\X$ has distinct $e$-th powers of eigenvalues, $\X^i$ is also non-derogatory. So, $\mc{C}(\X^i)$ is generated by the matrix $\X^i$, for all $i\ge 1$, by Lemma~\ref{lem:centralNak}((c) $\Rightarrow$ (b)). Hence $\mc{C}_P(\X)=\bigcup_{i\ge 1} \mc{C}(\X^i)=\kk[\X]$, which is commutative. Moreover, $\mc{C}_P(\X)=\mc{C}(\X)$ by Lemma~\ref{lem:centralNak}((b) $\Rightarrow$ (a)). 

For the other direction, suppose $\mc{C}_P(\X)$ is commutative. Then, $\mc{C}(\X)\subseteq \mc{C}_P(\X)$ is commutative. By way of contradiction, assume that $\X$ has two distinct eigenvalues $\lambda_1,\lambda_2$ such that $\lambda_1^e=\lambda_2^e$ for some $e \geq 1$. Then, the eigenvalue $\lambda:=\lambda_1^e=\lambda_2^e$ does not have geometric multiplicity 1 in $\X^e$. So,  by Lemma~\ref{lem:centralNak}((a) $\Rightarrow$ (c)), $\mc{C}(\X^e)\subseteq \mc{C}_P(\X)$ is not commutative,  a contradiction. Hence, $\X$ has distinct $e$-th powers of eigenvalues. 
\end{proof}

Another lemma that may be of use is the following.

\begin{lemma}\cite{DGKO} \label{prop:JNaka} Take any non-scalar matrix $\X\in M_n(\kk)$. If  $\X$ satisfies one of the conditions below:
\begin{enumerate}
\item[(i)] $\X$ is $\mc{C}$-equivalent to an idempotent matrix;
\item[(ii)] $\X$ is $\mc{C}$-equivalent to a square-zero matrix; or
\item[(iii)] $\X$ is similar to $\mathbb S \oplus \cdots \oplus \mathbb S$, where $\mathbb S$ is a companion matrix of an irreducible polynomial, such that there is no proper intermediate field between $\kk$ and $\kk[\mathbb S]$;
\end{enumerate} 
then  $\mathcal C_P(\X) = \mc{C}(\X)$. In this case,  $ \mc{C}_P(\X)$ is commutative if and only if $\mc{C}(\X)$ is commutative.
\end{lemma}

\begin{proof}
By \cite[Theorem~3.2]{DGKO},  any one of the conditions (i)-(iii) is equivalent to  $\X$ being {\it maximal}, which by definition means that for any $\M \in M_n(\kk)$ with $\mc{C}(\X) \subset \mc{C}(\M)$, we get that $\mc{C}(\X) = \mc{C}(\M)$. Thus, if $\X$ is maximal, then $\mc{C}(\X) = \mc{C}(\X^i)$ for all $i \geq 1$. Hence, $\mathcal C_P(\X) = \mc{C}(\X)$. The last statement is clear.
\end{proof}

\subsection{Application to Hopf coactions on $A(\E)$} \label{sec:application}
Recall from Lemma~\ref{lem:AS(E)} that $A(\E)$ is AS regular if and only if $\E \in GL_n(\kk)$ is congruent to a direct sum of matrices $\J_n$ and $\mathbb D_{2r}(q)$, with $q \neq 0,(-1)^{r+1}$.

\begin{proposition} \label{prop:cenNak}  
We have that $A(\E)$ is  central Nakayama if and only if $\E$ is congruent to 
\begin{equation} \label{eq:cenNakprime}
\begin{array}{c}\J_n\oplus \mathbb D_{2r_1}(q_1)\oplus\cdots \oplus \mathbb D_{2r_s}(q_s) \quad \text{ or }\quad \mathbb D_{2r_1}(q_1)\oplus\cdots \oplus \mathbb D_{2r_s}(q_s),
\end{array}
\end{equation}
with $q_i \neq 0,(-1)^{r_i+1}$ for all $i$, satisfying the following condition:
\begin{equation} \label{eq:cenNak}
\begin{array}{c}
\text{the set of numbers $\{q_1^{\pm 1},q_2^{\pm 1},\dots,q_s^{\pm 1}\}$ is distinct.}
\end{array} 
\end{equation}
Moreover, $A(\E)$ is power central Nakayama if and only if $\E$ is congruent to \eqref{eq:cenNakprime} so that:
\begin{equation} \label{eq:pcenNak}
\begin{array}{c}
\text{the set of numbers $\{q_1^{\pm e},q_2^{\pm e},\dots,q_s^{\pm e}\}$ is distinct, for any $e\geq1$.}
\end{array} 
\end{equation}
\end{proposition}

\begin{proof}
By Lemma~\ref{lem:NakAutom}(b), we have that $\mu_A|_{A_1}=-\E^{-1}\E^T$. By direct computation, we get that $-\J_n^{-1}\J_n^T$ and $-\mathbb D_{2r}(q)^{-1}\mathbb D_{2r}(q)^T$ are non-derogatory with eigenvalues $\{-1\}$ and $\{-q, -q^{-1}\}$, respectively. It is easy to see that a direct sum of non-derogatory matrices is non-derogatory if and only if the eigenvalues of all of the summands are distinct.  Thus, by Lemma~\ref{lem:centralNak}, $\mc{C}(\mu_A|_{A_1})$ is commutative if and only if $\E$ is congruent to~\eqref{eq:cenNakprime} and the set of numbers $\{q_1^{\pm 1},q_2^{\pm 1},\dots,q_s^{\pm 1}\}$ is distinct.

Moreover, to employ Lemma~\ref{lem:whenCP(X)com}, we need $\mu_A|_{A_1}=-\E^{-1}\E^T$ to have distinct $e$-th powers of eigenvalues, that is,  $-\J_n^{-1}\J_n^T$ and $-\mathbb D_{2r}(q)^{-1}\mathbb D_{2r}(q)^T$ to have distinct $e$-th powers of eigenvalues. Indeed, $\mu_A|_{A_1}$ has  distinct $e$-th powers of eigenvalues if and only if $\E$ is congruent to \eqref{eq:cenNakprime}, and \eqref{eq:pcenNak} is satisfied. Now we are done by Lemma~\ref{lem:whenCP(X)com}.
\end{proof}

 We provide some consequences below.

\begin{corollary} \label{cor:cocom}
Recall the hypotheses at the beginning of the section and in the Introduction. Let $A(\E)$ be an AS regular algebra of global dimension 2. Let $H$ be a Hopf algebra coacting on $A(\E)$ with homological codeterminant ${\sf D}\in H$. If one of the following conditions hold:
\begin{enumerate}
\item[(i)] $\E$ is congruent to \eqref{eq:cenNakprime} satisfying \eqref{eq:cenNak}, $H$ is involutory, and ${\sf D}$ is central; or
\item[(ii)] $\E$ is congruent to \eqref{eq:cenNakprime} satisfying \eqref{eq:pcenNak}, and ${\sf D}^m$ is central for some $m \geq 1$,
\end{enumerate}
then $H$ is cocommutative.  If, further, $H$ is finite dimensional, then $H$ must be a group algebra.
\end{corollary}

\begin{proof}
Apply Proposition~\ref{prop:cenNak} to get that $A:=A(\E)$ as above is central Nakayama in setting (i), and power central Nakayama in setting (ii).  Next, use \cite[Theorem~1.1]{Berger} to show that Hypothesis~\ref{hyp} holds. Finally, the result is established by Theorem~\ref{thm:centralNak}. \end{proof}

\begin{remark}
Note that $A(2,\E)$ is central Nakayama (resp., power central Nakayama) if and only if $A(2,\E) \not \cong A(2,\DD_2(\pm 1))$ (resp., $A(2,\E) \not \cong A(2,\DD_2(q))$ for $q \in \kk^{\times}$ a root of unity). Note that if $q^{2m} =1$, then we get by Example~\ref{ex:B(Aq)} that ${\sf D}^m$ is central. 

We point out that Corollary~\ref{cor:cocom} is consistent with \cite[Theorem~0.4]{CKWZ} for $A(2,\E)$. Namely, if $H$ is a finite dimensional Hopf algebra that coacts on $A(2,\E)$ inner-faithfully with trivial homological codeterminant, then $H$ is non-cocommutative (resp., and involutory) only when $\E = \mathbb D_2(q)$ for $q$ a root of unity (resp., when $q = -1$). 
On the other hand, the converse of Corollary~\ref{cor:cocom} fails as the cyclic group algebra, generated by Diag$(-1,-1)$, acts on {\it any} AS regular algebra $A(2,\E)$ inner-faithfully with trivial homological determinant by \cite[Theorem~5.2(b1)]{CKWZ}. 
\end{remark}

\begin{remark}
We can easily apply Corollary~\ref{cor:cocom} to get results on $H$-coactions on $A(\E) = A(3,\E)$. Here, $A(3,\E)$ is central Nakayama  if and only if $\E$ is congruent to $\J_3$ or $\J_1 \oplus \mathbb D_2(q)$, for $q \neq \pm 1$; and $A(3,\E)$ is power central Nakayama  if and only if $\E$ is congruent to $\J_3$ or $\J_1 \oplus \mathbb D_2(q)$, for $q \in \kk^{\times}$ a non-root of unity. See Corollary~\ref{cor:AS}(b).
\end{remark}

\section*{Acknowledgments}
The authors are grateful to Julien Bichon for providing careful comments on a preliminary version of this article, and thank Alexandru Chirvasitu, Dan Rogalski, and James Zhang for helpful correspondences. The authors are also thankful for Pavel Etingof and Debashish Goswami for agreeing to allow Example \ref{Ex:NoncomAction} to appear in the present paper. We would also like to thank the referee for valuable comments. C. Walton was supported by the National Science Foundation grant \#1550306.

\bibliography{HopfOneRelator}

\end{document}